%% file: main.tex
\documentclass[11pt]{article}

\usepackage[T1]{fontenc}
\usepackage{lmodern}

\usepackage{mathtools}
\usepackage{amssymb}
\usepackage{amsthm}

\usepackage{graphicx}
\usepackage{tikz}
\usepackage{caption}

\usepackage{microtype}

\usepackage[numbers,sort&compress]{natbib}

\usepackage{hyperref}
\usepackage{cleveref}

\usetikzlibrary{decorations.markings}

\theoremstyle{plain}
\newtheorem{theorem}{Theorem}[section]
\newtheorem{proposition}[theorem]{Proposition}
\newtheorem{lemma}[theorem]{Lemma}
\newtheorem{corollary}[theorem]{Corollary}

\theoremstyle{definition}
\newtheorem{definition}[theorem]{Definition}
\newtheorem{conjecture}[theorem]{Conjecture}

\theoremstyle{remark}
\newtheorem{remark}[theorem]{Remark}

\newcommand{\energy}{\mathcal{B}}                 
\newcommand{\excess}{\Lambda}                     
\newcommand{\disk}{\mathbb{D}}                    
\newcommand{\emb}{\mathcal{E}}                     
\newcommand{\minen}{m}                            
\newcommand{\turning}{\mathcal{T}}                
\newcommand{\sweepdens}{\omega}                   
\newcommand{\econn}{e_{\mathcal{A}}}              
\newcommand{\conncl}{\mathcal{A}}                 
\newcommand{\dgconst}{S}                          
\newcommand{\dipconst}{c_{*}}                     
\newcommand{\cstar}{c_{1}}                        
\newcommand{\quartrad}{\mathrm{q}}                
\newcommand{\dipcorr}{\mathrm{g}}                 
\newcommand{\ctwo}{c_{2}}                         
\newcommand{\bootstrap}{K}                        
\newcommand{\depthconst}{a_{*}}                   
\newcommand{\connarc}{\xi}                        
\newcommand{\bterm}{\mathrm{b}}                   

\DeclareMathOperator{\ind}{ind}
\DeclareMathOperator{\dist}{dist}
\DeclareMathOperator{\sgn}{sgn}
\DeclarePairedDelimiter{\abs}{\lvert}{\rvert}
\DeclarePairedDelimiter{\norm}{\lVert}{\rVert}

\title{Confined elastic wires of large length:\\ the sharp second-order asymptotics}
\author{L. Mugnai}
\date{\today}

\begin{document}

\maketitle

\begin{abstract}
We study the least bending energy $\minen_L$ of a closed embedded wire of prescribed length $L$ confined to the
closed unit disk.
It is classical that the bending energy of any closed curve in the disk is at least its length, with equality
only for multiply covered unit circles.
Such a circle is the energy limit of a stack of disjoint embedded circles, but not of a single embedded loop,
and the size of the resulting defect $\minen_L-L$ has remained open.
We prove that
\[
  \minen_L \;=\; L+\cstar\,L^{4/9}+O\bigl(L^{1/3}\bigr)\qquad\text{as }L\to\infty ,
\]
with a sharp constant $\cstar=5.223049\ldots$ given in closed form by two one-dimensional model problems.
In particular $\minen_L-L\asymp L^{4/9}$, which improves the previously known upper bound of order
$\sqrt{L}$ and provides the first lower bound beyond the trivial one.

This paper was produced with substantial assistance from large language models; 
\cref{sec:ai} sets out in detail what they contributed and what I verified.\end{abstract}

\input{sec1-intro}
\input{sec2-heuristics}
\input{sec3-notation}
\input{sec4-upper}

\input{sec5-lower}

\input{sec6-remarks}
\input{sec7-tool-disclosure}
\appendix
\input{appA-computations}

\bibliographystyle{plainnat}
\bibliography{refs}

\end{document}

%% file: sec1-intro.tex
\section{Introduction}
\label{sec:intro}

A closed wire of length $L$ is confined to the closed unit disk $\overline{\disk}\subset\mathbb{R}^{2}$ and adopts
the shape minimising its Bernoulli--Euler bending energy.
Writing $\emb_{L}$ for the class of embedded closed curves of class $W^{2,2}$, parametrised by arclength on
$\mathbb{R}/L\mathbb{Z}$ and with image in $\overline{\disk}$, the quantity of interest is
\[
  \minen_{L}\ :=\ \inf_{\gamma\in\emb_{L}}\ \energy(\gamma),
  \qquad \energy(\gamma):=\int_{0}^{L}\kappa^{2}\,ds .
\]
For $L\le2\pi$ the minimiser is the circle of radius $L/2\pi$ and $\minen_{L}=4\pi^{2}/L$, which decreases with
$L$; at $L=2\pi$ the wire fills $\partial\disk$.
Beyond that length the wire must fold into the disk, and the behaviour of $\minen_{L}$ for $L\gg1$ is the subject
of this paper.

Two elementary bounds frame the question.
For any closed curve in $\overline{\disk}$ one has $\int_{0}^{L}\abs{\kappa}\,ds\ge L$, with equality only for a
multiply covered unit circle: this is Chakerian's inequality, proved in \cite{Chakerian1962} by an
integral-geometric argument following F\'ary \cite{Fary1950}, and again in \cite{Chakerian1964} by the
integration of $\frac{d}{ds}\langle\gamma,\tau\rangle$ that we repeat in \cref{prop:budgets} below; see also
\cite[Prop.~6.1]{Sullivan2008}.
Cauchy--Schwarz then gives
\begin{equation}
  \label{eq:trivialbound}
  \energy(\gamma)\ \ge\ \frac{1}{L}\Bigl(\int_{0}^{L}\abs{\kappa}\,ds\Bigr)^{2}\ \ge\ L .
\end{equation}
Equality forces $\abs{\kappa}\equiv1$ and $r\equiv1$, that is, a unit circle covered $k=L/2\pi$ times, a closed
curve only when $L\in2\pi\mathbb{N}$.
That this curve is not embedded is not, by itself, the obstruction: it is the limit, in energy,
of $k$ disjoint embedded circles of radii close to $1$, that is, of a union of $k$ Jordan curves, and a stack
of independent wires would realise \cref{eq:trivialbound} up to $o(1)$.
What it is not is the $C^{1}$ limit of a \emph{single} embedded loop --- an embedded closed plane curve has
turning number $\pm1$, which passes to $C^{1}$ limits, whereas the $k$-fold circle has turning number $k$
(among the closed elasticae, classified in \cite{LangerSinger1984,AvvakumovKarpenkovSossinsky2013}, only the
once covered circle is such a limit \cite[Lem.~2.3]{Wojtowytsch2018}).
The defect in \cref{eq:trivialbound} is therefore the price of connectedness: it is incurred because the wire is
one embedded loop rather than a stack of independent circles, and its size is what we determine.

\begin{theorem}
  \label{thm:sharpmain}
  There are $C$ and $L_{0}$ such that, for all $L\ge L_{0}$,
  \[
    \minen_{L}\ =\ L+\cstar\,L^{4/9}+O\bigl(L^{1/3}\bigr),
    \qquad
    \cstar=\frac95\bigl(\dgconst^{2}\bigr)^{5/9}\Bigl(\frac{5\dipconst}{8\pi}\Bigr)^{4/9}=5.223049\ldots,
  \]
  where
  \[
    \dgconst=\int_{0}^{\pi}\sqrt{\sin t}\;dt=\frac{4\sqrt2\,\pi^{3/2}}{\Gamma(1/4)^{2}},
    \qquad
    \dipconst=\frac{32}{15}\,72^{1/4}.
  \]
  Moreover $\minen_{L}\le L+\cstar L^{4/9}+\ctwo L^{2/9}+C$ with
  $\ctwo=-1.4844067\ldots$
\end{theorem}

The defect is not a single term, and the theorem is the first step of a longer expansion.
The mechanism described in \cref{sec:heuristics} balances the cost $\asymp\delta^{-1}$ of one fold against the
cost of the $\asymp L$ dips of depth $\delta$ it forces, and the cost of a dip has an expansion in powers of
$\delta^{1/2}$; at the optimal depth $\delta\asymp L^{-4/9}$ the successive terms are of orders $L^{4/9}$,
$L^{2/9}$ and $L^{0}$, after which they vanish.
Adding the price of joining the two counter-rotating halves of the wire through the interior, which does not
grow with $L$, one is led to expect
\begin{equation}
  \label{eq:expansion}
  \minen_{L}\;=\;L+\cstar L^{4/9}+\ctwo L^{2/9}+c_{0}+o(1)\qquad(L\to\infty),
\end{equation}
in which $c_{0}$ collects the connector's contribution together with the $O(1)$ terms of the fold and of the
dips.
\Cref{thm:sharpmain} establishes the first two terms of \cref{eq:expansion}, and the coefficient of the third
from above only.

The two constants come from the two model problems of \cref{sec:heuristics}, and each is the value of an
explicit one-dimensional minimisation: $\dgconst^{2}$ is the least energy of a $\pi$-turn across a strip,
attained by a free elastica, and $\dipconst$ is the least cost of a unit dip, attained by a quartic profile.
The error $O(L^{1/3})$ comes entirely from the lower bound; the upper bound is accurate to $O(1)$ and already
exhibits the second term of \cref{eq:expansion}.
The same programme has been carried out at the opposite end of the range: Wojtowytsch
\cite[Thm.~1.1]{Wojtowytsch2018} shows that $\minen_{2\pi+\delta}=2\pi+\Theta\,\delta^{1/3}+o(\delta^{1/3})$ as
$\delta\downarrow0$, where the coefficient $\Theta\approx36.69$ is the value of an explicit one-dimensional
fourth-order obstacle problem with an integral constraint on the line \cite[Lem.~4.3]{Wojtowytsch2018}.
\Cref{thm:sharpmain} is the counterpart as $L\to\infty$, where two model problems rather than one are needed
because the fold and the dips are priced separately.
The constant $\ctwo$ is not obtained from below, and $c_{0}$ --- which contains, besides the $O(1)$ terms of
the fold and of the dips, the price of joining the two counter-rotating halves of the wire through the interior
of the disk --- not at all; \cref{sec:remarks} explains what is missing.

The radius of the container is fixed at $1$ for convenience only.
If $\minen^{R}_{L}$ denotes the same infimum for curves of length $L$ confined to a disk of radius $R$, then
$\gamma\mapsto\gamma/R$ is a bijection onto the competitors of length $L/R$ in the unit disk which multiplies
the bending energy by $R$, so that
\[
  \minen^{R}_{L}=\frac1R\,\minen_{L/R}
  \;=\;\frac{L}{R^{2}}+\cstar\,\frac{L^{4/9}}{R^{13/9}}+O\!\Bigl(\frac{L^{1/3}}{R^{4/3}}\Bigr)
\]
by \cref{thm:sharpmain}, valid whenever $L/R$ is large.

\subsection*{Background}

Tightly packed elastic filaments have been studied extensively in the physics literature, where the recurring
observation is that confinement produces \emph{spirals}.
Bou\'e et al.~\cite{Boue2006} identified the spiral as the elementary building block of two-dimensional packing
and described the ``yin-yang'' pattern, an \emph{S}-shaped curve at the core surrounded by spiralling layers; Alben
\cite{Alben2021} computed the frictionless packing of a ring in a shrinking disk and found the same morphology,
a spiral of nearly circular turns together with a single sharp bend trapped against the wall.
Wires injected into a planar circular cavity show the same spiral morphology in experiment and in
discrete-element simulation, together with a friction-dependent transition to cascading loops
\cite{DonatoGomesDeSouza2003,BoueKatzav2007,StoopWittelHerrmann2008,VetterWittelStoopHerrmann2013}.
The corresponding problem for a rod compressed inside a rectangular channel has been analysed in detail, and
there too a transition is observed from distributed wrinkles to a localised spiral, the so-called yin-yang
state \cite{Deboeuf2024,Abramian2026,Wang2026}.

Those results concern tightly packed wires, in which friction and the excluded volume of a wire of positive
thickness do genuine work.
Here the bending energy is the whole of the model and neither is charged for: excluded volume becomes the
dominant effect once the $\asymp L$ turns of the bundle can no longer be nested at negligible cost, and the
asymptotics below describe the complementary regime, in which the wire is thin enough for the fold to be the
leading obstruction.
For the bending energy penalised by thickness rather than by a container see
\cite{GerlachReiterVonDerMosel2017}, where multiply covered circles again appear as the limit configurations,
and \cite{GerlachVonDerMosel2011} for the associated packing problem.

As for $\minen_{L}$, the asymptotics of \cite{Alben2021} for a ring of length $2\pi$ in a disk of radius $R_{w}$,
rescaled so that the container becomes the unit disk, give a bending energy $L+O(\sqrt{L})$: the spiral itself
contributes only the baseline $L$, and the entire excess is the price of the one sharp fold.
On the mathematical side, Wojtowytsch \cite{Wojtowytsch2018} proved the matching upper bound
$\minen_{L}\le L+C\sqrt{L}$ in the large-length regime, no better lower bound than \cref{eq:trivialbound} being
available.

\Cref{thm:sharpmain} shows that $\sqrt{L}$ is not optimal, and the reason is instructive.
In the configurations just described the fold is given room by displacing the whole bundle of turns inward,
uniformly, by the fold's width $\delta$: this costs $\delta$ per turn, so $\asymp L\delta$ in total, against a
fold cost $\asymp\delta^{-1}$, and the balance $\delta=L^{-1/2}$ returns $\sqrt{L}$ --- this is precisely the
balance struck in \cite[\S5]{Wojtowytsch2018}, where the optimal constant is explicitly left undetermined.
The improvement is to make the displacement \emph{local}: if each turn dips to depth $\delta$ only over an
angular window of width $\delta^{1/4}$, its cost falls from $\delta$ to $\delta^{5/4}$, and balancing
$\delta^{-1}$ against $L\delta^{5/4}$ gives $\delta=L^{-4/9}$ and an excess of order $L^{4/9}$.
\Cref{sec:heuristics} explains why $\delta^{1/4}$ is the optimal window width and why no other arrangement does
better; the harder half of \cref{thm:sharpmain} is the matching lower bound, which must exclude every competitor
rather than every competitor of a prescribed shape.

The infimum $\minen_{L}$ is in general not attained in $\emb_{L}$.
By \cite[Prop.~2.1]{DondlMugnaiRoger2011} and \cite[Lem.~2.6]{Wojtowytsch2018}, if $\overline{\emb_{L}}$
denotes the $W^{2,2}$-closure of $\emb_{L}$, then
\[
  \minen_{L}\;=\;\min_{\overline{\emb_{L}}}\ \energy ,
\]
and a minimiser may have self-contacts and touch $\partial\disk$, but has a unique tangent line at every
point.
For large $L$ minimisers are expected to have self-contacts; a proof is outside the scope of this paper.
The elements of $\overline{\emb_{L}}$ are described, though not completely characterised, in terms of
$W^{2,2}$-immersions and generalised curves in
\cite{BellettiniDalMasoPaolini1993,delladio1997special,BellettiniMugnai2004,bellettini2007varifolds}; we do not
pursue such a characterisation, because the route followed below works directly with elements of
$\emb_{L}$.
Other constrained elasticae are treated in \cite{DayrensMasnouNovaga2015,MiuraYoshizawa2024}, the
Li--Yau-type inequality of \cite{MuellerRupp2021,Miura2023} bounds the multiplicity of a $W^{2,2}$ curve in
terms of $\energy\cdot L$, the survey \cite{MantegazzaPludaPozzetta2021} covers the elastic energy of curves in
general, and \cite{BartelsWeyer2022} give a convergent scheme for computing confined elasticae in convex
domains.

The contrast with the three-dimensional analogue studied by M\"uller and R\"oger \cite{MullerRoger2014} ---
surfaces of sphere type of prescribed area, embedded in the unit ball, minimising the Willmore energy --- is
instructive, because there the two features that fail in the plane are both present.
The Willmore energy is scale invariant, and nested spheres may be merged into a single embedded surface by
catenoidal necks, which carry no energy at all, being minimal; so the energy can be made to scale exactly like
the first-order term, and the defect stays bounded.
In the plane, as already noted, $\energy(\lambda\gamma)=\lambda^{-1}\energy(\gamma)$ is not scale invariant, and
merging two nested circles requires reversing the tangent: a fold confined to width $\delta$ costs at least a constant
multiple of $\delta^{-1}$, and \cref{thm:sharpmain} measures what that obstruction is worth.

\subsection*{Plan}

\Cref{sec:heuristics} is an informal account of the mechanism fixing the exponent, together with a description
of the competitor; nothing later depends on it except the elementary trade-off \cref{eq:generaltradeoff}.
\Cref{sec:notation} fixes notation.
\Cref{sec:upper} constructs the competitor and proves the upper bound.
\Cref{sec:lower} proves the lower bound: it develops an identity expressing the excess of any arc of the wire,
up to a boundary term, through pointwise nonnegative densities, shows that reversals of the sense of rotation
are few, selects a fold on the outer envelope of the wire where nothing lies above it, and prices that fold
against the passages it forces underneath.
\Cref{sec:remarks} collects remarks and open problems.

%% file: sec2-heuristics.tex
\section{Why the exponent is \texorpdfstring{$4/9$}{4/9}}
\label{sec:heuristics}

This section is informal, and no proof depends on it apart from the trade-off \cref{eq:generaltradeoff}, a
one-line calculus exercise quoted for motivation.
All statements are to be read up to constants and lower-order terms.

\subsection{The excess, and why topology forces a fold}
\label{sec:heur-fold}

Write $\excess(\gamma):=\energy(\gamma)-L$ for the \emph{excess}.
By \cref{eq:trivialbound}, $\excess\ge0$ with equality only for the multiply covered unit circle, so $\excess$
measures the deviation from that circle, in two senses which \cref{sec:budgets} makes quantitative: almost all
of the length must lie close to $\partial\disk$ and run almost tangentially, and $\abs\kappa$ must be close to
$1$ in mean square.
The picture is a bundle of about $L/2\pi$ nearly circular strands hugging the boundary, which costs nothing to
leading order; the entire excess is the price of assembling it into one embedded loop.

That price is not zero because of orientation.
A strand hugging the boundary sweeps polar angle at unit rate, so the \emph{unsigned} angular travel is about
$L$, whereas the \emph{signed} total is $2\pi\ind_{\gamma}(0)\in\{0,2\pi\}$ for a Jordan curve.
Hence about half the length runs counterclockwise and half clockwise: somewhere the wire turns around, and we
call such a turning point a \emph{fold}.
Folds are forced before any energy consideration, and the two costs they generate are what the exponent
balances.

\subsection{The two costs}
\label{sec:heur-costs}

Let the fold sit at distance $\delta$ from $\partial\disk$.
Reversing the sense of travel rotates the tangent by $\pi$, and inside a channel of width $\delta$ this must be
done over a length $\asymp\delta$, so by Cauchy--Schwarz the fold costs $\gtrsim\delta^{-1}$; the sharp
constant is $\dgconst^{2}$, where $\dgconst$ is the constant denoted $c_{0}$ by Deckelnick and Grunau
\cite[Lem.~4]{DeckelnickGrunau2007} in their study of the one-dimensional Willmore equation, see
\cref{lem:depthchart,prop:head}.
Cheap folds are wide folds --- but a wide fold is an obstacle: at its angle it occupies every depth between $0$
and $\delta$, so each of the remaining $\asymp L$ passages, unable to cross it, must dip to depth $\delta$
there.

The cost of one dip is where the exponent is made.
Writing a strand near the boundary as a polar graph $r=1-u(\theta)$ with $u\ge0$ small, the expansion of length
and curvature recorded in \cref{sec:appA} gives excess
\begin{equation}
  \label{eq:secondorder}
  \int\Bigl(2u+(u'')^{2}\Bigr)d\theta\;+\;\text{lower order},
\end{equation}
the linear term paying for the length lost by moving inward, the quadratic one for the extra bending.
Both are essential: the first alone is minimised by not dipping, the second by dipping infinitely gently.
Inserting $u(\theta)=\delta^{a}v(\theta\delta^{-b})$, a fixed profile $v\ge0$ of compact support
scaled independently in amplitude and window, turns \cref{eq:secondorder} into
$C\delta^{\min(a+b,\,2a-3b)}$.
The amplitude is prescribed by the fold that must fit underneath, the window width is free, and equalising the
two exponents gives $b=\tfrac14a$: with $a=1$, a dip of depth $\delta$ is optimally spread over
an angular window of width $\delta^{1/4}$ and costs $\delta^{5/4}$.
The exponent is therefore not an artefact of the ansatz but the scaling of the model problem
$\inf\{\int(2u+(u'')^{2}):u\ge0,\ \max u\ge\delta\}$, whose minimiser is a quartic profile; the same
competition between a linear term and $(u'')^{2}$ under the constraint $u\ge0$ governs the adhesive obstacle
problem of \cite{Miura2016}.

\subsection{The trade-off}
\label{sec:heur-dip}

The two costs pull against each other, and in general form
\begin{equation}
  \label{eq:generaltradeoff}
  \min_{\delta>0}\Bigl(\frac1\delta + L\,\delta^{\beta}\Bigr)
  \;=\;(1+\beta)\,\beta^{-\frac{\beta}{1+\beta}}\,L^{\frac{1}{1+\beta}}
  \;\asymp\; L^{\frac{1}{1+\beta}},
\end{equation}
attained at $\delta_{*}=(\beta L)^{-1/(1+\beta)}$.
At the optimum the two terms are comparable, so the value is $\asymp\delta_{*}^{-1}$: the excess is, to
leading order, the price of the fold alone.
With $\beta=\tfrac54$ this gives $\delta\asymp L^{-4/9}$ and excess of order $L^{4/9}$, the content of
\cref{thm:sharpmain}.
The comparison is instructive: charging a dip only for depth and tilt, ignoring the bending it requires, gives
$\beta=\tfrac32$ and the weaker exponent $2/5$; the whole difference lies in the term $(u'')^{2}$ of
\cref{eq:secondorder}.
Finer structure follows from the same picture.
Carrying the expansion one order further produces a correction $-3\int(u')^{2}$ to the per-dip cost, of
relative size $\delta^{1/2}$, so the successive terms $L\delta^{5/4}$, $L\delta^{7/4}$, $L\delta^{9/4}$ are of
orders $L^{4/9}$, $L^{2/9}$, $L^{0}$, and the next tends to zero; the fold contributes
$\delta^{-1}(1+O(\delta))$, at the first and the last of these orders.
With the $O(1)$ price of the connector this is \cref{eq:expansion}.

\subsection{The competitor, and the lower bound}
\label{sec:heur-competitor}

\Cref{fig:construction} shows the curve realising the upper bound.
A \emph{bundle} of $2N$ nearly circular passages hugs $\partial\disk$ and carries almost all the length; over
one angular window of width $\delta^{1/4}$ the passages \emph{dip} to depth $\asymp\delta$, freeing a
lens-shaped region against the boundary; inside it sits the \emph{fold}, a U-turn of width $\delta$ traversed
once, joined to the innermost passage by a short arc; and an \emph{S-connector} routed through the empty
interior joins the two innermost ends of the bundle at a cost independent of $\delta$.
Adding the costs gives $\excess\lesssim\delta^{-1}+L\delta^{5/4}$, and \cref{eq:generaltradeoff} with
$\delta=L^{-4/9}$ yields the upper bound.

That this cannot be beaten is \cref{sec:lower}, and the argument follows the same two costs.
The difficulty is that an arbitrary competitor need not be organised as a bundle with one fold, and that a fold
might be \emph{shielded}: if other passages ran between it and the boundary, it could be wide without forcing
anyone to dip.
The observation which removes this is that the fold may be chosen on the \emph{outer envelope}, at a point
outermost along its own ray.
Such a fold exists: otherwise the envelope would be a closed polar graph traversed by the wire, hence, being a
closed connected subset of a Jordan curve, the whole of it, so the wire would sweep angle in one sense only,
contradicting the balance above.
Nothing lies above an envelope fold, so every passage across its ray lies below it and both horns apply at
once.
Two further points occupy most of \cref{sec:lower}: the passages must be shown to dip in \emph{distinct}
windows, so that their costs add, which is where the scarcity of reversals enters; and the per-dip cost must be
extracted from an arbitrary competitor rather than from a polar graph, since near a fold the wire need not be a
graph at all, which is done by working with the depth as a function of arclength, for which an exact identity
replaces \cref{eq:secondorder}.

\begin{figure}
  \centering
  \includegraphics[width=0.8\textwidth]{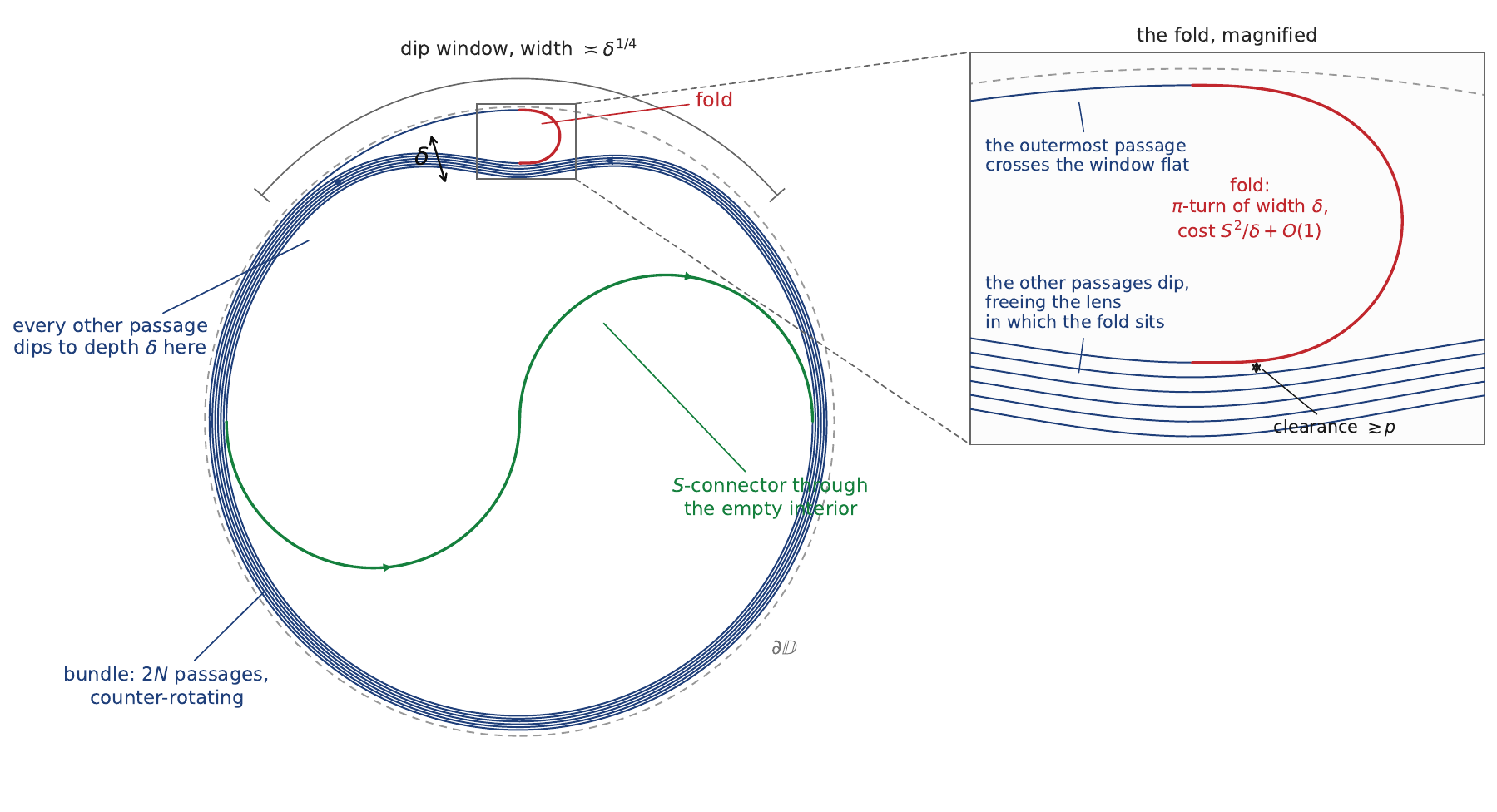}
  \caption{The competitor of \cref{sec:upper}, drawn with the fold region exaggerated.
  The bundle of $2N\asymp L$ passages hugs $\partial\disk$; over an angular window of width $\delta^{1/4}$ it
  dips to depth $\delta$, and the fold---a U-turn traversed once---occupies the lens so freed between the dip and
  the boundary.
  The two counter-rotating halves of the bundle are joined through the empty interior by the S-connector, whose
  endpoints are antipodal and whose prescribed tangents are therefore parallel.
  Balancing the cost $\delta^{-1}$ of the U-turn against the total cost $N\,\delta^{5/4}$ of the dips gives
  $\delta\asymp L^{-4/9}$.}
  \label{fig:construction}
\end{figure}

%% file: sec3-notation.tex
\section{Notation}
\label{sec:notation}

Throughout, $\disk=\{z\in\mathbb{R}^{2}:\abs{z}<1\}$ and curves are parametrised by arclength on $[0,L]$ and
identified with maps on $\mathbb{R}/L\mathbb{Z}$.
We write $\tau=\gamma'$, $\nu=i\tau$ for the unit normal, $\kappa$ for the signed curvature, so that
$\tau'=\kappa\nu$, and
\[
  \energy(\gamma)=\int_{0}^{L}\kappa^{2}\,ds,\qquad
  \turning(\gamma)=\int_{0}^{L}\abs{\kappa}\,ds,\qquad
  \excess(\gamma)=\energy(\gamma)-L .
\]
We identify $\mathbb{R}^{2}$ with $\mathbb{C}$ and write $\det(a,b)=\operatorname{Im}(\bar a b)$, so that
$\det(a,ib)=\langle a,b\rangle$.
The class of competitors is
\[
  \emb_{L}:=\bigl\{\gamma\in W^{2,2}(\mathbb{R}/L\mathbb{Z};\overline{\disk}):\
    \abs{\gamma'}\equiv 1,\ \gamma\ \text{injective}\bigr\},
  \qquad \minen_{L}:=\inf_{\emb_{L}}\energy .
\]
Every competitor is oriented positively, i.e.\ $\ind_{\gamma}\equiv+1$ on the bounded component of the complement
of the trace; this is no restriction, since reversing the parametrisation changes neither $\energy$ nor the
length, and every density estimated in \cref{sec:lower} is invariant under the simultaneous change of sign of
$\kappa$ and $\tau$ that a reversal produces.
Two scalar functions of the arclength recur throughout: the radius $r:=\abs{\gamma}$ and the \emph{depth}
\[
  h:=1-r\in[0,1] .
\]
The depth is $1$-Lipschitz, so $\abs{h'}\le1$ almost everywhere, and it is of class $W^{2,2}$ on every set
where $r\ge c>0$, the map $x\mapsto\abs x$ being smooth away from the origin; a parameter with $\gamma(s_{0})=0$ produces a corner,
$h(s)=1-\abs{s-s_{0}}+o(\abs{s-s_{0}})$.
Away from the Lipschitz bound, every statement involving $h'$ or $h''$ is confined to $\{h\le\tfrac1{10}\}$,
where $r\ge\tfrac9{10}$.
With an abuse of notation we use the same letter for the depth of a point, $h(x):=1-\abs{x}$ for
$x\in\overline{\disk}$, so that $h(s)=h(\gamma(s))$.
The elementary inequality
\begin{equation}
  \label{eq:depthbound}
  h\ \le\ h\,(2-h)\ =\ 1-r^{2}\qquad\text{on }\overline{\disk}
\end{equation}
is used repeatedly to pass from \cref{eq:B1} to bounds on $\int h$.

We write $a\asymp b$ when each of $a$ and $b$ is bounded by a fixed multiple of the other, the implied
constants being absolute unless stated otherwise.

\subsection{The radial frame}
\label{sec:frames}

Every quantity in this paper, metric or angular, is referred to the origin; no estimate involves
$\dist(0,\gamma([0,L]))$.
Wherever $r>0$ we write $e_{r}:=\gamma/r$ and $e_{\theta}:=ie_{r}$ for the radial and angular directions,
$\varphi$ for a continuous polar angle about the origin, and we define the angle $\psi$ by the decomposition
\begin{equation}
  \label{eq:frame}
  \tau\ =\ \cos\psi\,e_{\theta}\ +\ \sin\psi\,e_{r},
\end{equation}
so that
\begin{equation}
  \label{eq:framerelations}
  h'=-\sin\psi,\qquad \varphi'=\frac{\cos\psi}{1-h},\qquad \kappa=\varphi'-\psi' .
\end{equation}
Up to a shift, $\psi$ is the classical \emph{polar tangential angle} of a plane curve, the angle between the
position vector and the tangent; see \citet{Miura2021} for its geometry, and in particular for its monotonicity
along curves of monotone curvature.
In this frame the \emph{sweep density}
\begin{equation}
  \label{eq:sweepdef}
  \sweepdens\ :=\ \det(\gamma,\tau)\ =\ -\langle\gamma,\nu\rangle\ =\ r\cos\psi
\end{equation}
is the rate at which the curve sweeps angle about the origin, weighted by $r^{2}$: by
\cref{eq:framerelations}, $\sweepdens=r^{2}\varphi'$ wherever $r>0$.
It is defined everywhere, also through the origin, and it is the variable in which folds are counted in
\cref{sec:lower-scarcity}.

%% file: sec4-upper.tex
\section{The upper bound}
\label{sec:upper}

Throughout this section $\Delta\in(0,\Delta_{0})$ is the dip depth, $\Delta_{0}$ a small absolute constant,
$v$ is the quartic profile of \cref{lem:quartic} below, supported on $[-\quartrad,\quartrad]$ with
$\quartrad=72^{1/4}$, and $w:=\quartrad\Delta^{1/4}$ is the half-width of the dip window.
The number of turns is $N\in\mathbb{N}$, and the \emph{pitch} is
\begin{equation}
  \label{eq:pitchchoice}
  p:=\Delta^{5},
  \qquad\text{with the standing assumption}\qquad N\Delta^{9/4}\le1 ,
\end{equation}
so that the total depth of the bundle satisfies $Np\le\Delta^{11/4}\le\Delta$, and
$\Delta^{5}N^{2}=(N\Delta^{9/4})^{2}\Delta^{1/2}\le1$; in the regime $N\asymp L$,
$\Delta\asymp L^{-4/9}$ of \cref{thm:upper} the standing assumption holds with room to spare.
The exponent is dictated by two competing constraints made precise in
\cref{prop:upperconstruction,rem:whydelta}: the pitch must exceed the $O(\Delta^{11/2})$ penetration depth of
the fold, and $N^{2}p$ must vanish in the regime $N\asymp L$, $\Delta\asymp L^{-4/9}$ of
\cref{thm:upper}; any $p$ with $\Delta^{11/2}\ll p\ll N^{-2}$ works, and \cref{eq:pitchchoice} is a
convenient representative.
Constants denoted $C$ are absolute, and may change from line to line; they never depend on $\Delta$, on $L$,
or on the number of passages.

\subsection{The competitor}
\label{sec:upper-competitor}

The curve is assembled from a bundle of nearly circular passages, a fold, and a connector through the interior.
We describe the three parts in polar coordinates $z=(1-u)e^{i\theta}$, so that $u$ is the depth.

\emph{The bundle.}
Let $N\in\mathbb{N}$, let the pitch $p=\Delta^{5}$ be as in \cref{eq:pitchchoice}, and set
\[
  \theta_{A}:=2\pi N+\tfrac\pi2,\qquad \theta_{B}:=2\pi N-\tfrac\pi2,
\]
so that $\theta_{A}-\theta_{B}=\pi$ and $\theta_{A}+\theta_{B}=4\pi N$.
Let
\[
  \beta(\theta):=\Delta\,v\!\left(\frac{\theta\ \mathrm{mod}\ 2\pi}{\Delta^{1/4}}\right),
  \qquad \theta\ \mathrm{mod}\ 2\pi\in[-\pi,\pi),
\]
with $v$ the profile of \cref{lem:quartic}, so that $\beta$ is supported in the windows
$\abs{\theta-2\pi k}\le w:=\quartrad\Delta^{1/4}$, has $\beta(2\pi k)=\Delta$, and vanishes at the edge of each
window together with its first two derivatives.
Define, for $\theta\in[0,\theta_{A}]$ and $\theta\in[0,\theta_{B}]$,
\begin{equation}
  \label{eq:spirals}
  u_{A}(\theta):=\frac p2+\frac{p\,\theta}{2\pi}+\beta(\theta),
  \qquad
  u_{B}(\theta):=\frac p4+\frac{p\,\theta}{2\pi}+\beta(\theta)\,\mathbb 1_{\{\theta\ge w\}} .
\end{equation}
Thus every turn carries the dip $\beta$ except the outermost turn of $\Gamma_{B}$, which crosses its window
flat; the two formulas for $u_{B}$ agree on $[w,2\pi-w]$, where $\beta\equiv0$, so $u_{B}$ is as smooth as
$\beta$.
The graphs $\Gamma_{A}$ and $\Gamma_{B}$ are disjoint and each is embedded, since the baselines are ordered and
$\beta\ge0$ is common to all turns but one, which is the outermost.
The \emph{outer} endpoints, at $\theta=0$, are at depths $\tfrac p2+\Delta$ --- the bottom of the outermost dip
of $\Gamma_{A}$ --- and $\tfrac p4$; the \emph{inner} endpoints, at $\theta_{A}$ and $\theta_{B}$, are
antipodal, lie at angular distance $\tfrac\pi2>w$ from the nearest window, and are at depths differing by
$\tfrac{3p}4$.
The curve traverses $\Gamma_{A}$ with $\theta$ increasing and $\Gamma_{B}$ with $\theta$ decreasing, so that the
two carry opposite orientations, as they must.
Counting window centres, $\Gamma_{A}$ carries $N$ complete dips and $\Gamma_{B}$ carries $N-1$, so the bundle
carries $2N-1$ in all.

\emph{The fold.}
Over the window $\abs\theta\le w$ every turn but the outermost has depth at least $\tfrac p2+\beta(\theta)$,
so the lens $\{\abs\theta\le w,\ \tfrac p4\le u\le\tfrac p2+\beta(\theta)\}$ is free of the bundle.
The fold is placed inside it: it joins the outer endpoint of $\Gamma_{A}$ to that of $\Gamma_{B}$, rising in
depth by $\Delta+\tfrac p4$ while reversing the sense of rotation (\cref{lem:fold}).

\emph{The connector.}
The bundle lies in the annulus $\{1-pN-\Delta\le\abs{z}\le1\}$, so the disk of radius $\tfrac12$ about the
origin is free of it, and both inner endpoints are the deepest points of their graphs at their respective
angles.
The connector joins them through the interior (\cref{lem:connector}).
Because the two inner endpoints are antipodal, and because $\Gamma_{A}$ is traversed with $\theta$ increasing
while $\Gamma_{B}$ is traversed with $\theta$ decreasing, the tangents to be matched at the two ends of the
connector are \emph{parallel}: this is what makes the connector an \emph{S}-shaped arc of bounded curvature rather
than a hairpin.

\subsection{The excess of a polar graph}
\label{sec:upper-graph}

All bundle estimates reduce to the following expansion, proved in \cref{sec:appA}.

\begin{lemma}
  \label{lem:graphexcess}
  Let $I\subset\mathbb{R}$ be an interval and $u\in W^{2,\infty}(I)$ with
  $\norm{u}_{\infty}+\norm{u'}_{\infty}\le\tfrac18$, and let $\Gamma$ be the polar graph
  $\theta\mapsto(1-u(\theta))e^{i\theta}$, $\theta\in I$, of length
  $\mathcal{H}^{1}(\Gamma):=\int_{I}\sqrt{(1-u)^{2}+(u')^{2}}\,d\theta$ (counted with multiplicity, $I$ being
  allowed to be longer than $2\pi$).
  Then, with a constant $C$ independent of $u$ and $I$,
  \begin{align*}
    \Biggl|\ \int_{\Gamma}\kappa^{2}\,ds-\mathcal{H}^{1}(\Gamma)
      &-\int_{I}\Bigl(2u+2u''+(u'')^{2}\Bigr)d\theta\ \Biggr|\\
      &\le\; C\int_{I}\Bigl(u^{2}+(u')^{2}+\abs{u\,u''}+\abs{u''}^{3}\Bigr)d\theta .
  \end{align*}
\end{lemma}

Two features of this formula drive the construction.
The term $2u''$ integrates to $2[u']_{\partial I}$, hence contributes nothing over a full turn or over a window
at whose endpoints $u'$ vanishes.
The remaining terms $2u$ and $(u'')^{2}$ are the two costs described in \cref{sec:heur-dip}: moving inward loses
length, and doing so requires bending.

\subsection{The pieces}
\label{sec:upper-pieces}

\begin{lemma}
  \label{lem:bundle}
  The bundle has total length
  \[
    4\pi N-\int_{0}^{\theta_{A}}u_{A}\,d\theta-\int_{0}^{\theta_{B}}u_{B}\,d\theta
      +O\bigl(N\Delta^{7/4}+\Delta^{10}N\bigr),
  \]
  in which the two integrals sum to a quantity $\asymp N\Delta^{5/4}$, and excess at most
  $C\bigl(N\Delta^{5/4}+\Delta^{5}N^{2}\bigr)$.
\end{lemma}

\begin{proof}
  Note first that \cref{eq:pitchchoice} gives $\norm{u_{A}}_{\infty}\le\tfrac p2+pN+\Delta\le3\Delta\le\tfrac18$
  for $\Delta_{0}$ small, so \cref{lem:graphexcess} applies.
  Outside the windows $\beta\equiv0$, so there $u_{A}$ is affine with $u_{A}'=p/2\pi$ and $u_{A}''=0$, and by
  \cref{lem:graphexcess} the excess of $\Gamma_{A}$ there is at most
  $\int\bigl(2u_{A}+Cu_{A}^{2}+C(u_{A}')^{2}\bigr)d\theta$.
  The affine part $u_{A}-\beta$ satisfies $0\le u_{A}-\beta\le2pN$ on all of $[0,\theta_{A}]$, so it contributes
  $\int_{0}^{\theta_{A}}\bigl(p+p\theta/\pi\bigr)d\theta\le Cp N^{2}=C\Delta^{5}N^{2}$ to the first term, while
  $(u_{A}-\beta)^{2}\le4p^{2}N^{2}$ and $(u_{A}')^{2}=p^{2}/4\pi^{2}$ integrate, over a length $O(N)$, to at
  most $Cp^{2}N^{3}=C\Delta^{5}N^{2}\cdot\Delta^{5}N\le C\Delta^{5}N^{2}$, since
  $\Delta^{5}N\le\Delta^{11/4}\le1$ by \cref{eq:pitchchoice}; the contribution of $\beta$, and of its coupling
  to the affine part, is accounted for below.

  On a window we have $u_{A}=\text{affine}+\beta$ with $\abs{\beta}\le\Delta$,
  $\abs{\beta'}\le C\Delta w^{-1}$ and $\abs{\beta''}\le C\Delta w^{-2}$, while $u'$ takes the same value
  $p/2\pi$ at both endpoints, so the term $2u''$ integrates to zero there.
  Hence one window contributes at most
  \begin{align*}
    \int_{\abs{\theta}\le w}\Bigl(2\beta+(\beta'')^{2}\Bigr)d\theta
    &+ C\int_{\abs{\theta}\le w}\Bigl(\beta^{2}+(\beta')^{2}+\abs{\beta\beta''}+\abs{\beta''}^{3}\Bigr)d\theta\\
    &\le\; C\bigl(\Delta w+\Delta^{2}w^{-3}\bigr)+C\Delta^{7/4},
  \end{align*}
  the remainder terms being bounded, in order, by $\Delta^{2}w$, $\Delta^{2}w^{-1}$, $\Delta^{2}w^{-1}$ and
  $\Delta^{3}w^{-5}$, each at most $C\Delta^{7/4}$ for $w=\quartrad\Delta^{1/4}$, at which value the leading
  bracket equals $C\Delta^{5/4}$; the terms coupling $\beta$ to the affine part are smaller still, that part
  being at most $2pN\le2\Delta^{11/4}$.
  There are $N$ complete windows on $\Gamma_{A}$, centred at $2\pi k$ for $1\le k\le N$, and $N-1$ on
  $\Gamma_{B}$, for $1\le k\le N-1$; in addition $\Gamma_{A}$ carries the half-window $[0,w]$ of its
  outermost dip, whose cost is $O(\Delta^{5/4})$ by the same estimate.
  The bounds for $\Gamma_{B}$ are identical, which gives the stated excess.

  For the length, $\sqrt{(1-u)^{2}+(u')^{2}}=(1-u)+O\bigl((u')^{2}\bigr)$, so
  \[
    \int_{0}^{\theta_{A}}\!\!\sqrt{(1-u_{A})^{2}+(u_{A}')^{2}}\,d\theta
    \ =\ \theta_{A}-\int_{0}^{\theta_{A}}u_{A}\,d\theta+O\bigl(p^{2}N+N\Delta^{2}w^{-1}\bigr),
  \]
  and similarly for $\Gamma_{B}$; the error is $O\bigl(N\Delta^{7/4}+\Delta^{10}N\bigr)$, and the term
  $N\Delta^{7/4}$ is genuinely present --- it is the length carried by $(\beta')^{2}$, of order $L^{2/9}$ in
  the regime of \cref{thm:upper} --- but harmless: it perturbs the passage count at fixed length by
  $O(L^{2/9})$ and hence the dip total by $O(L^{2/9}\Delta^{5/4})=O(L^{-1/3})$.
  Here $\int_{\abs\theta\le w}\beta\,d\theta=\Delta^{5/4}\int v$, so that
  $\int_{0}^{\theta_{A}}u_{A}=\pi pN^{2}+N\Delta^{5/4}\int v+O(\Delta)$, and likewise for $\Gamma_{B}$;
  since $\int v>0$ is a fixed number and $\Delta^{5}N^{2}\le\Delta^{5/4}N$ by \cref{eq:pitchchoice}
  --- indeed $N\Delta^{15/4}\le\Delta^{3/2}\le1$ --- the two integrals sum to a quantity
  $\asymp N\Delta^{5/4}$.
  It is no accident that the length deficit has the same order as the excess: dipping is precisely a trade of
  length for bending.
\end{proof}

\begin{lemma}
  \label{lem:strip}
  For $\delta\in(0,\tfrac1{10})$ let $\Phi_{\delta}(x,y):=(1-\delta y)\,e^{i\delta x}$, which maps the strip
  $\mathcal{S}:=\mathbb{R}\times[0,3]$ into
  $\overline{\disk}$, the coordinate $y$ measuring depth, so that $y=0$ is $\partial\disk$; the map sends
  depth and angle exactly, $1-\abs{\Phi_{\delta}(x,y)}=\delta y$ and $\arg\Phi_{\delta}(x,y)=\delta x$.
  Let $\xi\subset[-C_{0},C_{0}]\times[0,3]$ be a $C^{1,1}$ arc with $\abs{\kappa_{\xi}}\le C_{0}$ and
  $\mathcal{H}^{1}(\xi)\le C_{0}$.
  Then $\Phi_{\delta}$ multiplies lengths by $\delta\bigl(1+O(\delta)\bigr)$, transforms curvature by
  \[
    \kappa_{\Phi_{\delta}(\xi)}=\delta^{-1}\bigl(\kappa_{\xi}+O(\delta)\bigr),
  \]
  and consequently
  $\int_{\Phi_{\delta}(\xi)}\kappa^{2}ds=\delta^{-1}\!\int_{\xi}\kappa^{2}ds+O(1)$, all constants depending
  only on $C_{0}$.
  The transformation of curvature is additive and not multiplicative: a horizontal segment has
  $\kappa_{\xi}=0$, while its image is an arc of the circle of radius $1-\delta y$, of curvature close to
  $1$.
\end{lemma}

\begin{proof}
  A direct computation gives $D\Phi_{\delta}=\delta\,R(x)\bigl(I+O(\delta)\bigr)$ with $R(x)$ orthogonal, which is the
  first assertion.
  For the curvature, write $\Phi_{\delta}=1+\delta\Phi_{0}$ with
  $\Phi_{0}(x,y)=\bigl(e^{i\delta x}-1\bigr)/\delta-y\,e^{i\delta x}$, which converges in $C^{2}$ on the
  bounded set $[-C_{0},C_{0}]\times[0,3]$ to the Euclidean motion $(x,y)\mapsto ix-y$; hence for an arc $\connarc$ of bounded $C^{1,1}$ norm
  $\kappa_{\Phi_{\delta}(\connarc)}=\delta^{-1}\bigl(\kappa_{\connarc}+O(\delta)\bigr)$.
\end{proof}

\begin{lemma}
  \label{lem:connector}
  There is an embedded $C^{1,1}$ arc $\Xi$ joining the inner endpoint of $\Gamma_{A}$ to that of $\Gamma_{B}$
  with matching tangents, meeting the bundle only at those endpoints, with
  \[
    \mathcal{H}^{1}(\Xi)\le C,\qquad \int_{\Xi}\kappa^{2}\,ds\le C .
  \]
\end{lemma}

\begin{proof}
  Write $q_{A}$ and $q_{B}$ for the two inner endpoints and $\tau_{A},\tau_{B}$ for the unit tangents to be
  matched, taken in the sense in which the competitor traverses the arms.
  They are antipodal up to an error $O(pN)$ and lie at depths at most $2pN\le2\Delta$, by
  \cref{eq:pitchchoice}.
  The tangent to $\Gamma_{A}$ at $q_{A}$ is $e_{\theta}(\theta_{A})+O(p)$, while $\Gamma_{B}$ is traversed with
  $\theta$ decreasing, so $\tau_{B}=-e_{\theta}(\theta_{B})+O(p)=e_{\theta}(\theta_{A})+O(p)$: the two prescribed
  tangents are parallel and point in the same direction.
  Consequently the data $(q_{A},\tau_{A})$ and $(q_{B},\tau_{B})$ are, up to an error $O(pN)$, those of two
  antipodal boundary points with parallel tangents, which are joined inside $\overline{\disk}$ by an $S$-shaped
  arc of length and curvature bounded by absolute constants; take for $\Xi$ any such arc, perturbed by
  $O(pN)$ to match the data exactly.
  No small-radius turn is forced, because the two arms run in opposite senses along the circle and so
  separate immediately.
  The arc may be chosen inside $\{\abs{z}\le\tfrac12\}$ apart from two short segments near its endpoints, hence
  disjoint from the bundle except at $q_{A}$ and $q_{B}$: indeed at the angles $\theta_{A}$ and $\theta_{B}$ the
  deepest point of the bundle is the corresponding endpoint, because $u_{A}(\theta_{A}-2\pi j)$ and
  $u_{B}(\theta_{B}-2\pi j)$ decrease in $j$.
\end{proof}

\subsection{Assembly}
\label{sec:upper-assembly}

Each of the two pieces is now taken to be the minimiser of its model problem: the dip profile is the quartic
of \cref{lem:quartic}, and the fold is the free elastica of \cref{lem:fold}.
The resulting competitor has an excess matching the lower bound of \cref{thm:lowersharp} to leading order,
and in addition exhibits the second term of \cref{eq:expansion}.

\subsubsection*{The optimal dip}

\begin{lemma}
  \label{lem:quartic}
  Let $v$ be defined on $[-\quartrad,\quartrad]$, $\quartrad=72^{1/4}$, by $v=\alpha t^{3}-t^{4}/24$ with
  $t=\quartrad-\abs{x}$ and $\alpha=\quartrad/18$, extended by zero.
  Then $v\ge0$, $v(0)=1$, and $v=v'=v''=0$ at $\pm\quartrad$, so that $v\in C^{2,1}(\mathbb{R})$ after
  extension by zero; the third derivative jumps there, by $\quartrad/3$, and
  \[
    \int\bigl(2v+(v'')^{2}\bigr)=\dipconst=\frac{4\quartrad^{5}}{135},
    \qquad
    \int (v')^{2}=\dipcorr:=\frac{\quartrad^{7}}{1890}=0.9416078\ldots
  \]
\end{lemma}

\begin{proof}
  The stated identities are elementary computations, and $v'(0)=0$ forces $\alpha=\quartrad/18$ while
  $v(0)=\quartrad^{4}/72=1$ forces $\quartrad^{4}=72$.
  Writing $t=\quartrad-\abs{x}$ one has $v_{xx}=6\alpha t-\tfrac{t^{2}}2$ and $v_{xxx}=\pm(t-6\alpha)$, whence
  $v_{xx}\to0$ and $v_{xxx}\to\mp\quartrad/3$ as $t\to0$: the profile is $C^{2}$ but not $C^{3}$ at the edge of
  its support.
  This is the correct free-boundary behaviour for a fourth-order obstacle problem, in which the multiplier
  $2+2v''''\ge0$ may carry an atom on the contact set, and it is compatible with the $W^{2,\infty}$ hypothesis
  of \cref{lem:graphexcess}.
  The minimiser of the model problem of \cite[Lem.~4.3]{Wojtowytsch2018} is of class $C^{2,1}\setminus C^{3}$
  for exactly this reason; for the general theory, in which the multiplier is a measure and $W^{3,\infty}$ is
  the optimal regularity, see \cite{CaffarelliFriedman1979,DallAcquaDeckelnick2018}.
\end{proof}

\begin{remark}
  \label{rem:quarticoptimal}
  The profile is not merely convenient but optimal: $v$ minimises $\int_{I}(2u+(u'')^{2})$ among
  $u\in W^{2,2}(I)$ with $u\ge0$ and $\max u\ge1$, on any interval $I$ containing a maximum point at
  distance at least $\quartrad$ from $\partial I$.
  This is \cref{lem:dipsharp} with $A=1$, and it is the same calibration that prices the dips from below in
  \cref{sec:lower}; the upper bound of the present section uses only the value $\dipconst$ of the functional
  at $v$, so nothing here depends on \cref{sec:lower}.
  Obstacle problems for the elastic energy of graphs, of which \cref{lem:dipsharp} is an explicit instance, are
  studied in \cite{DallAcquaDeckelnick2018,Mueller2019,Yoshizawa2021} and, for the $p$-elastic energy, in
  \cite{DallAcquaMuellerOkabeYoshizawa2024}; the loss of regularity at the free boundary is analysed in detail in
  \cite{Yoshizawa2021}.
\end{remark}

The dip carried by each strand is $\beta(\theta):=\Delta\,v\bigl(\theta\Delta^{-1/4}\bigr)$, of depth $\Delta$
and angular width $2\quartrad\Delta^{1/4}$.
Its cost requires \cref{lem:graphexcess} one order further than stated there.

\begin{lemma}
  \label{lem:dipcost}
  Let $u=u_{0}+\beta$ with $u_{0}$ affine and
  \[
    \norm{u_{0}}_{\infty}\le C_{0}\Delta,\qquad \abs{u_{0}'}\le C_{0}\Delta ,
  \]
  and let $I$ be a window at whose endpoints $\beta$ vanishes together with $\beta'$, so that $u'$ takes
  equal values there.
  Then, with constants depending on $C_{0}$, the excess carried by the polar graph of $u$ over $I$, in excess
  of that of $u_{0}$, equals
  \[
    \int_{I}\Bigl(2\beta+(\beta'')^{2}-3(\beta')^{2}\Bigr)d\theta+O\bigl(\Delta^{9/4}\bigr)
    \;=\;\dipconst\Delta^{5/4}-3\dipcorr\,\Delta^{7/4}+O\bigl(\Delta^{9/4}\bigr).
  \]
\end{lemma}

\begin{proof}
  With $a,b,n$ as in \cref{sec:appA}, the exact density is
  $\tfrac{1-a}{\sqrt a}+\tfrac{2b}{a^{3/2}}+\tfrac{b^{2}}{a^{5/2}}$, and expanding one order beyond
  \cref{lem:graphexcess} gives
  \[
    2u+2u''+(u'')^{2}+\bigl((u')^{2}+4uu''+u^{2}\bigr)+O\bigl(\text{integrand }O(\Delta^{2})\bigr).
  \]
  Over the window $\int2u''=0$, while
  $\int4uu''=\bigl[4uu'\bigr]_{\partial I}-4\int(u')^{2}$ with
  $\bigl[4uu'\bigr]_{\partial I}=4u_{0}'\bigl(u_{0}(b)-u_{0}(a)\bigr)=4(u_{0}')^{2}\abs I=O(\Delta^{9/4})$,
  whence the coefficient $-3$; here the hypothesis $\abs{u_{0}'}\le C_{0}\Delta$ is used, an affine baseline of
  slope only $O(\Delta^{3/4})$ producing a boundary term of the same order $\Delta^{7/4}$ as the term being
  computed.
  The terms in which $u_{0}$ meets $\beta$ are likewise admissible: $\int2u_{0}\beta$ and
  $\int3u_{0}(\beta'')^{2}$ are $O\bigl(\norm{u_{0}}_{\infty}\Delta^{5/4}\bigr)=O(\Delta^{9/4})$, while
  $\int2u_{0}'\beta'$ and $\int(u_{0}')^{2}\beta''$ vanish, $\beta$ and $\beta'$ doing so at $\partial I$.
  The remaining terms integrate to $O(\Delta^{9/4})$ because $\abs{\beta}\le\Delta$,
  $\abs{\beta'}\le C\Delta^{3/4}$, $\abs{\beta''}\le C\Delta^{1/2}$ on a window of width
  $2\quartrad\Delta^{1/4}$.
  The evaluation is \cref{lem:quartic} after the scaling $\beta=\Delta v(\cdot\,\Delta^{-1/4})$.
\end{proof}

\subsubsection*{The optimal fold}

\begin{lemma}
  \label{lem:fold}
  Let $\Sigma_{1}$ be half a period of the free elastica $\kappa^{2}=\dgconst^{2}\sin\vartheta$
  (for the classical theory of elasticae see \cite{LangerSinger1984}) in the strip of \cref{lem:strip}, where $\vartheta\in[0,\pi]$ is the tangent angle, normalised to unit rise and placed so
  that both endpoints lie on the line $x=0$; explicitly,
  \begin{equation}
    \label{eq:foldchartstrip}
    x(\vartheta)=-\frac{2}{\dgconst}\sqrt{\sin\vartheta},
    \qquad
    y(\vartheta)=\frac{1}{\dgconst}\,\widehat G(\vartheta)
      =\frac{1}{\dgconst}\int_{0}^{\vartheta}\sqrt{\sin t}\,dt .
  \end{equation}
  Then $\Sigma_{1}$ is a graph over the depth coordinate $y$, hence embedded; it rises in depth by exactly
  $1$; its two tangents are horizontal and opposite, so that it reverses the sense of rotation; its curvature
  vanishes at both ends; it lies in $\{-\tfrac2\dgconst\le x\le0\}$, of horizontal extent
  $\tfrac2\dgconst=0.8346$, with zero net horizontal displacement; and
  \[
    \int_{\Sigma_{1}}\kappa^{2}\,ds=\dgconst^{2},
    \qquad \mathcal{H}^{1}(\Sigma_{1})=2.1884\ldots
  \]
  For $d\in(0,\tfrac1{10})$, the transplant $\Sigma_{d}:=\Phi_{d}(\Sigma_{1})\subset\overline{\disk}$ of
  \cref{lem:strip} has the exact polar description $\theta=d\,x(\vartheta)$, $u=d\,y(\vartheta)$: it rises in
  depth by exactly $d$, occupies the angular sector $-\tfrac{2d}\dgconst\le\theta\le0$, and costs
  \[
    \int_{\Sigma_{d}}\kappa^{2}\,ds\ =\ \frac{\dgconst^{2}}{d}+O(1) .
  \]
\end{lemma}

\begin{proof}
  Minimising $\int(\vartheta')^{2}ds$ with $\vartheta$ running from $0$ to $\pi$ and free length, under the
  constraint that the rise $\int\sin\vartheta\,ds$ be prescribed, gives $2\kappa'=\mu\cos\vartheta$ and hence
  the first integral $\kappa^{2}=\mu\sin\vartheta$, the additive constant vanishing because $\kappa\to0$ at
  the free ends.
  Then $\int\kappa^{2}ds=\int\kappa\,d\vartheta=\sqrt\mu\int_{0}^{\pi}\sqrt{\sin\vartheta}\,d\vartheta
  =\sqrt\mu\,\dgconst$ and rise $=\int\sin\vartheta\,ds=\mu^{-1/2}\dgconst$, so unit rise forces
  $\mu=\dgconst^{2}$, and the energy is $\dgconst^{2}$.
  The chart \cref{eq:foldchartstrip} follows from $dy/d\vartheta=\sin\vartheta/\kappa=\sqrt{\sin\vartheta}/\dgconst$
  and $dx/d\vartheta=-\cos\vartheta/\kappa=-\cos\vartheta/(\dgconst\sqrt{\sin\vartheta})$, integrated from
  $\vartheta=0$; in particular $x(0)=x(\pi)=0$ and $x$ attains its minimum $-\tfrac2\dgconst$ at
  $\vartheta=\tfrac\pi2$, which gives the extent, the sector, and the vanishing net displacement.
  The length is $\dgconst^{-1}\int_{0}^{\pi}(\sin\vartheta)^{-1/2}d\vartheta=B(\tfrac14,\tfrac12)/\dgconst
  =2.1884\ldots$
  Since $\sin\vartheta\ge0$ throughout, the depth is monotone along $\Sigma_{1}$, which is therefore a graph
  over it.
  For the transplant, \cref{lem:strip} sends depth and angle exactly, $u=d\,y$ and $\theta=d\,x$, and applies
  to $\Sigma_{1}$, whose curvature is bounded by $\dgconst$ and whose length and horizontal extent are bounded
  by absolute constants; it gives
  $\int_{\Sigma_{d}}\kappa^{2}=d^{-1}\!\int_{\Sigma_{1}}\kappa^{2}+O(1)$.
\end{proof}

\begin{remark}
  \label{rem:foldlegs}
  The two legs of $\Sigma_{d}$ lie at \emph{different} depths, separated by the full rise $d$: a $\pi$-turn
  with monotone tangent angle changes the depth by $\int\sin\vartheta\,ds>0$.
  A fold whose two legs sit at the same depth must therefore turn by more than $\pi$,
  and cannot achieve $\dgconst^{2}/d$.
  This is what dictates the arrangement below: the strand delivered by the shallow end of the fold, at the
  outer baseline $\tfrac p4$, is the outermost one, and it is the one strand that does not dip.
  The same one-dimensional problem --- least bending energy of a $\pi$-turn of prescribed transverse span and
  free length --- arises in the theory of stiff polymers confined to nanochannels, where it is treated by
  elliptic integrals \cite[\S II]{Odijk2006}.
\end{remark}

\subsubsection*{The competitor, and the sharp upper bound}

\begin{proposition}
  \label{prop:upperconstruction}
  Let $\Delta\in(0,\Delta_{0})$ and $N$ satisfy \cref{eq:pitchchoice}, and let $\Gamma_{A},\Gamma_{B}$ be the
  interleaved spirals of \cref{eq:spirals}, every turn carrying the dip $\beta$ except the outermost, at
  baseline $\tfrac p4$, which crosses its window flat.
  Join the bottom of the outermost dip of $\Gamma_{A}$ to the outer end of the flat strand by the arc
  $\Sigma_{d}$ of \cref{lem:fold} with $d:=\Delta+\tfrac p4$, smoothing the two junctions as described in the
  proof, and join the two inner ends by the arc of \cref{lem:connector}.
  The result is a closed embedded $W^{2,2}$ curve, of length $L=F(N,\Delta)$ with
  $\abs{F(N+1,\Delta)-F(N,\Delta)-4\pi}\le C\Delta^{5/4}$, and
  \[
    \excess\ \le\ \frac{\dgconst^{2}}{\Delta}
      +\bigl(2N-1\bigr)\Bigl(\dipconst\Delta^{5/4}-3\dipcorr\Delta^{7/4}\Bigr)+C .
  \]
\end{proposition}

\begin{proof}
  \emph{Geometry of the fold.}
  The rise from the outer end $(\theta,u)=(0,\tfrac p4)$ of the flat strand to the bottom
  $(0,\tfrac p2+\Delta)$ of the outermost dip of $\Gamma_{A}$ is exactly
  $\tfrac p2+\Delta-\tfrac p4=d$, and by \cref{lem:fold} the fold placed between them has the exact polar
  description
  \begin{equation}
    \label{eq:foldchart}
    \theta=d\,x(\vartheta)=-\frac{2d}{\dgconst}\sqrt{\sin\vartheta},
    \qquad
    u=\frac p4+d\,y(\vartheta),
    \qquad \vartheta\in[0,\pi] :
  \end{equation}
  it lies in the sector $-\tfrac{2d}\dgconst\le\theta\le0$, of angular width $\tfrac{2d}\dgconst=0.8346\,d$,
  on the opposite side of $\theta=0$ from the half-window $[0,w]$, and at depths in
  $[\tfrac p4,\tfrac p2+\Delta]$.

  \emph{Embeddedness of the bundle.}
  The dipping strands are the graphs $u_{k}=\text{baseline}_{k}+\beta$ with baselines strictly increasing by
  steps of at least $\tfrac p4$ and $\beta$ common to all, hence ordered and disjoint; the flat strand has
  $u\le\tfrac p4+\tfrac{p\theta}{2\pi}<\tfrac p2+\tfrac{p\theta}{2\pi}\le u_{k}$ at every common angle, since
  $\beta\ge0$, so it lies strictly outside all of them.

  \emph{The junction estimate.}
  Fix a point of $\Sigma_{d}$, of angle $\theta\in[-\tfrac{2d}\dgconst,0]$ and depth $\tfrac p4+dy$.
  At that angle, that is at $\theta+2\pi\in[2\pi-\tfrac{2d}\dgconst,2\pi]$, the competitor's strands are the
  turns $k\ge1$ of both spirals, whose baselines are at least $\tfrac p4+p$; the dip $\beta$ being common to
  all of them, each such strand lies at depth at least
  $\tfrac{5p}4+\beta(\theta)=\tfrac{5p}4+\Delta v(z)$ with
  \[
    z:=\theta\Delta^{-1/4}=d\,x\,\Delta^{-1/4},
    \qquad
    \abs z\le\tfrac54\Delta^{3/4}\abs x\le\tfrac{5}{2\dgconst}\,\Delta^{3/4}\le\Delta^{1/2},
  \]
  using $d\le\tfrac54\Delta$, $\abs x\le\tfrac2\dgconst$ and $\Delta\le\Delta_{0}$; in particular
  $\abs z\le w\Delta^{-1/4}=\quartrad$, so the point of the strand is inside its dip window.
  It suffices to show
  \begin{equation}
    \label{eq:junction}
    \Delta\bigl(y-v(z)\bigr)\ \le\ \tfrac p2
    \qquad\text{on }\Sigma_{d} ;
  \end{equation}
  indeed then, since $y\le1$,
  \[
    \Bigl(\tfrac{5p}4+\Delta v(z)\Bigr)-\Bigl(\tfrac p4+d\,y\Bigr)
    \ =\ p-\Delta\bigl(y-v(z)\bigr)-\tfrac p4\,y\ \ge\ \tfrac p4 ,
  \]
  and every point of the fold is strictly shallower than every strand at its angle, with margin $\tfrac p4$.

  On the shallow half $\vartheta\le\tfrac\pi2$ one has $y=\widehat G(\vartheta)/\dgconst\le\tfrac12$, while
  the Taylor bound
  \[
    v(z)\ \ge\ 1-\frac{\quartrad^{2}}{12}\,z^{2}-\frac{\quartrad}{9}\,\abs z^{3},
  \]
  valid because $v(0)=1$, $v'(0)=0$, $v''(0)=-\tfrac{\quartrad^{2}}6$ and
  $\norm{v'''}_{\infty}=\tfrac{2\quartrad}3$, gives $v(z)\ge\tfrac34$ for $\Delta_{0}$ small; hence
  $y-v\le-\tfrac14$ and \cref{eq:junction} is clear.
  On the deep half $\vartheta\in[\tfrac\pi2,\pi]$ the fold is cubically flat at its deep end: with
  $\varphi:=\pi-\vartheta\in[0,\tfrac\pi2]$ and $\sin\varphi=\sin\vartheta$,
  \[
    1-y\ =\ \frac1{\dgconst}\int_{0}^{\varphi}\sqrt{\sin t}\,dt
    \ \ge\ \frac{(\sin\varphi)^{3/2}}{3\,\dgconst}
    \ =\ \frac{\dgconst^{2}}{24}\,\abs x^{3},
  \]
  the middle inequality holding because both sides vanish at $\varphi=0$ and the derivative of the left-hand
  side, $\sqrt{\sin\varphi}/\dgconst$, dominates that of the right-hand side,
  $\sqrt{\sin\varphi}\,\cos\varphi/2\dgconst$; the final equality is \cref{eq:foldchart}.
  Combining the two bounds and $\abs z\le\tfrac54\Delta^{3/4}\abs x$, $\abs x\le\tfrac2\dgconst$,
  \[
    \begin{aligned}
      \Delta\bigl(y-v(z)\bigr)
      &\ \le\ \frac{25\,\quartrad^{2}}{192}\,\Delta^{5/2}x^{2}
        +\frac{125\,\quartrad}{288\,\dgconst}\,\Delta^{13/4}x^{2}
        -\frac{\dgconst^{2}}{24}\,\Delta\,\abs x^{3}\\
      &\ \le\ 2\,\Delta^{5/2}x^{2}-\frac{\dgconst^{2}}{24}\,\Delta\,\abs x^{3}
    \end{aligned}
  \]
  for $\Delta_{0}$ small.
  If $\abs x\ge x_{1}:=\tfrac{48}{\dgconst^{2}}\Delta^{3/2}$ the negative term dominates and
  $\Delta(y-v)\le0$; if $\abs x<x_{1}$, then
  $\Delta(y-v)\le2\Delta^{5/2}x_{1}^{2}=2\bigl(\tfrac{48}{\dgconst^{2}}\bigr)^{2}\Delta^{11/2}
  =C\Delta^{1/2}\,p\le\tfrac p2$ for $\Delta_{0}$ small, by \cref{eq:pitchchoice}.
  In both cases \cref{eq:junction} holds.
  The fold moreover lies in $\overline\disk$, its depths being at least $\tfrac p4>0$, and is embedded, being
  a graph over the depth by \cref{lem:fold}; and it is disjoint from the connector, whose excursions outside
  $\{\abs z\le\tfrac12\}$ occur at angles within $O(\Delta)$ of $\theta_{A}$ and $\theta_{B}$, hence at
  angular distance at least $\tfrac\pi4$ from the sector of the fold.

  \emph{Smoothing the junctions.}
  At $\theta=0$ both ends of $\Sigma_{d}$ are horizontal with vanishing curvature, whereas the strands they
  meet have slope $u'=p/2\pi$ there, the dip contributing $\beta'(0)=0$: positions match, but the tangents
  differ by an angle $\alpha\le p$.
  At each of the two junctions, replace the corner by an arc of curvature $1$ and length $\alpha$ tangent to
  both branches: this costs $\int\kappa^{2}\le\alpha\le p$ per junction, restores the $C^{1}$ regularity with
  bounded curvature there, and displaces the curve by at most $C\alpha^{2}\le Cp^{2}$, far below the margin
  $\tfrac p4$ of the junction estimate and the gap $\tfrac p4$ between consecutive strands, so embeddedness is
  preserved.
  Away from the junctions each piece is of class $W^{2,2}$, so the closed concatenation is an embedded
  $W^{2,2}$ curve.

  \emph{Energy.}
  The baselines of the strands are affine with $\norm{u_{0}}_{\infty}\le3\Delta$, as in the proof of
  \cref{lem:bundle}, and slope $u_{0}'=p/2\pi\le\Delta$, so \cref{lem:dipcost} applies to each dip.
  By \cref{lem:bundle,lem:dipcost} the $2N-1$ complete dips contribute
  $(2N-1)\bigl(\dipconst\Delta^{5/4}-3\dipcorr\Delta^{7/4}\bigr)+O\bigl(N\Delta^{9/4}\bigr)$, and the
  half-dip, the affine parts and the depth of the bundle contribute
  $O\bigl(\Delta^{5/4}\bigr)+O\bigl(\Delta^{5}N^{2}\bigr)=O(1)$, and
  $O\bigl(N\Delta^{9/4}\bigr)=O(1)$, both by \cref{eq:pitchchoice}.
  By \cref{lem:fold} the fold contributes
  $\dgconst^{2}/d+O(1)=\dgconst^{2}/\Delta+O\bigl(p\Delta^{-2}\bigr)+O(1)=\dgconst^{2}/\Delta+O(1)$, since
  $p\Delta^{-2}=\Delta^{3}$; the smoothing arcs contribute $O(p)$, and the connector $O(1)$.
  Finally, by the length formula of \cref{lem:bundle}, adding one turn to each spiral changes the length by
  $4\pi$ up to per-turn terms $O(pN)+O(\Delta^{5/4})+O(\Delta^{7/4})$, each at most $C\Delta^{5/4}$ by
  \cref{eq:pitchchoice}, which is the stated increment property of $F$.
\end{proof}

\begin{theorem}
  \label{thm:upper}
  There are $C$ and $L_{0}$ such that, for every $L\ge L_{0}$,
  \[
    \minen_{L}\ \le\ L+\cstar\,L^{4/9}+\ctwo\,L^{2/9}+C,
    \qquad
    \ctwo:=-\frac{3\dipcorr\,\depthconst^{7/4}}{2\pi}=-1.4844067\ldots,
  \]
  where $\depthconst:=\bigl(8\pi\dgconst^{2}/5\dipconst\bigr)^{4/9}=1.9788993\ldots$ and
  $\Delta:=\depthconst L^{-4/9}$.
\end{theorem}

\begin{proof}
  Fix $\Delta=\depthconst L^{-4/9}$ and let $N$ be the largest integer with $F(N,\Delta)\le L$; then
  $N=\tfrac{L}{4\pi}\bigl(1+o(1)\bigr)$, so
  $N\Delta^{9/4}=\tfrac{\depthconst^{9/4}}{4\pi}\bigl(1+o(1)\bigr)\le1$ for $L\ge L_{0}$, and
  \cref{eq:pitchchoice} is in force.
  Rather than dilating, which would cost an amount of order one, adjust the length exactly by extending both arms through a
  common angle $t$: write $\ell(N,t)$, $t\in[0,2\pi)$, for the length of the configuration in which both arms
  are continued beyond their $N$ complete turns through the further angle $t$, so that $\ell$ is continuous
  and increasing in $t$ with $\ell(N,0)=F(N,\Delta)$ and $\ell(N,2\pi^{-})=F(N+1,\Delta)$.
  By the choice of $N$ one has $\ell(N,0)\le L<\ell(N,2\pi^{-})$, so $\ell(N,t)=L$ for some $t$.
  The extension preserves the antipodality of the inner ends, leaves the fold untouched, and adds length at
  depth $O(Np)=O(\Delta^{11/4})$, so that it costs $o(1)$ and creates at most two further dips, one per arm,
  of cost $O(\Delta^{5/4})$ each.
  With $2N-1=\tfrac{L}{2\pi}+O(L\Delta^{5/4})$ and $\dipconst\Delta^{5/4}=O(L^{-5/9})$, the count perturbs the
  total by $O(L^{-1/9})$, and \cref{prop:upperconstruction} gives
  \[
    \excess\ \le\ \frac{\dgconst^{2}}{\Delta}+\frac{L}{2\pi}\Bigl(\dipconst\Delta^{5/4}-3\dipcorr\Delta^{7/4}\Bigr)+C .
  \]
  The first two terms are minimised at $\Delta=\depthconst L^{-4/9}$, with value $\cstar L^{4/9}$, and the third is
  $\ctwo L^{2/9}$ at that value of $\Delta$.
\end{proof}

\begin{remark}
  \label{rem:matching}
  \Cref{thm:lowersharp,thm:upper} together are \cref{thm:sharpmain}.
  The upper bound is the sharper of the two: it already exhibits the term $\ctwo L^{2/9}$ of
  \cref{eq:expansion}, with $\ctwo<0$, whereas the lower bound is accurate only to $O(L^{1/3})$ for the reason
  given in \cref{rem:whyLthird}.
\end{remark}

\begin{remark}
  \label{rem:whydelta}
  The choice $\Delta\asymp L^{-4/9}$ is the one dictated by \cref{eq:generaltradeoff}: it equalises the cost
  $\Delta^{-1}$ of the single fold with the total cost $N\Delta^{5/4}\asymp L\Delta^{5/4}$ of the dips.
  The pitch $p$ plays no role in that balance; it is squeezed from two sides, since the bundle-depth excess
  requires $N^{2}p\to0$ while the junction estimate of \cref{prop:upperconstruction} requires
  $p\gg\Delta^{11/2}$, and any $p$ in that range does, \cref{eq:pitchchoice} being one convenient choice.
  This is possible precisely because the two counter-rotating halves of the bundle are joined through the empty
  interior rather than across the pitch: a fold placed between two adjacent passages would cost $p^{-1}$, and no
  choice of $p$ makes both $p^{-1}$ and the excess $\asymp N^{2}p$ caused by the bundle's depth small.
\end{remark}

%% file: sec5-lower.tex
\section{The lower bound}
\label{sec:lower}

This section proves the lower half of \cref{thm:sharpmain}, with the sharp constant: this is
\cref{thm:lowersharp} below.
The proof is carried out not on all of $\emb_{L}$ but on the subclass
\begin{equation}
  \label{eq:restrictedclass}
  \emb_{L}^{\bootstrap}\;:=\;\bigl\{\gamma\in\emb_{L}:\ \excess(\gamma)\le\bootstrap\,L^{4/9}\bigr\},
\end{equation}
where $\bootstrap:=\cstar+2$.
By \cref{thm:upper} the competitor constructed there has excess at most $(\cstar+1)L^{4/9}$ for
$L\ge L_{0}$, since $\ctwo<0$; so $\emb_{L}^{\bootstrap}$ is nonempty, and the remaining unit gives the class
the slack needed for the approximation in \cref{rem:avoidorigin} below, whose competitors are admitted at the
cost of an arbitrarily small increase of the energy.
This is a genuine reduction of the competitor class, and not a case distinction: it is legitimate precisely
because the upper bound is already available.

\begin{remark}
  \label{rem:reduction}
  The class $\emb_{L}^{\bootstrap}$ is nonempty for $L\ge L_{0}$, since the curve produced by \cref{thm:upper}
  belongs to it, and every curve outside it has excess exceeding $\bootstrap L^{4/9}$, hence energy larger than
  that of any member; so
  \[
    \minen_{L}\;=\;\inf_{\gamma\in\emb_{L}^{\bootstrap}}\energy(\gamma),
  \]
  and it suffices to bound $\excess$ from below on $\emb_{L}^{\bootstrap}$.
\end{remark}

Accordingly we assume throughout that
\begin{equation}
  \label{eq:bootstrap}
  \excess=\excess(\gamma)\le \bootstrap\,L^{4/9}.
\end{equation}
All constants $c,C$ below, and the threshold $L_{0}$, are allowed to depend on $\bootstrap$; no estimate uses more
about $\excess$ than \cref{eq:bootstrap}, so $\bootstrap$ enters only through such constants.

\begin{remark}
  \label{rem:avoidorigin}
  We may, and do, also assume throughout that
  \begin{equation}
    \label{eq:avoidorigin}
    0\notin\gamma([0,L]) ,
  \end{equation}
  by \cref{lem:density}.
  Indeed a curve with $\excess>(\bootstrap-1)L^{4/9}=(\cstar+1)L^{4/9}$ satisfies every lower bound proved in
  this section outright; and for a
  curve with $\excess\le(\bootstrap-1)L^{4/9}$ the approximants of \cref{lem:density} with $\eta\le1$ have
  excess at most $\bootstrap L^{4/9}$, so they lie in $\emb_{L}^{\bootstrap}$, avoid the origin, and satisfy
  $\excess(\gamma)\ge\excess(\gamma')-\eta$; letting $\eta\downarrow0$ transfers to $\gamma$ any lower bound
  proved under \cref{eq:avoidorigin}.
  Under \cref{eq:avoidorigin} the depth $h$ is of class $C^{1}\cap W^{2,2}$ on all of
  $\mathbb{R}/L\mathbb{Z}$; but no estimate below depends on the value of $\dist(0,\gamma([0,L]))$, which the
  approximation does not control --- quantitative bounds involving $h''$ or $\varphi'$ are always localised to
  depth sets such as $\{h\le\tfrac1{10}\}$.
\end{remark}

Finally, by \cref{eq:avoidorigin} the polar angle $\varphi$ of \cref{eq:framerelations} admits a global
continuous lift, fixed once and for all; \emph{all} angular bookkeeping below --- the envelope radius
$R(\theta)$, folds, blocking and the passage count --- refers to the origin and to this lift.

\Cref{sec:lift,sec:budgets} assemble the tools: an angular lift, and a master identity expressing the excess of
any arc, up to a boundary term, as the integral of pointwise nonnegative densities, so that it may be localised.
\Cref{sec:lower-scarcity} then shows that sign alternations of the sweep density at a definite amplitude are
few, so that the curve splits into $O(\excess)$ arcs of constant orientation.
\Cref{sec:lower-envelope} selects a fold on the outer envelope, where by construction nothing lies above it.
\Cref{sec:lower-head,sec:lower-dip} prepare the pricing of that fold and of the passages forced beneath it,
and \cref{sec:lower-sharp} carries both out with the sharp constants and balances them.
\Cref{sec:upper} is independent of this section.

The two notions on which everything rests, the angular lift and the fold, are made precise in
\cref{def:lift} below.
\subsection{The angular lift and folds}
\label{sec:lift}

Let us start by making precise two notions recurring throughout the rest of the paper.

\begin{definition}[Angular lift, fold]
  \label{def:lift}
  Let $\gamma\in\emb_{L}$ satisfy \cref{eq:avoidorigin} and let $r=\abs\gamma$.
  \begin{enumerate}
    \item An \emph{angular lift} of $\gamma$ is a continuous $\varphi:[0,L]\to\mathbb{R}$ with
      \[
        \gamma(s)=r(s)\,e^{i\varphi(s)}\qquad\text{for every }s\in[0,L],
      \]
      that is, a continuous choice of the polar angle along the curve.
    \item A parameter $s$ is a \emph{fold} if $\varphi'(s)=0$.
  \end{enumerate}
\end{definition}

An angular lift exists and is unique up to an additive constant in $2\pi\mathbb{Z}$, and is of class
$W^{2,2}$; this is \cref{lem:lift} below, which also records the two identities the lift satisfies.
We fix one such lift once and for all and call it \emph{the} polar angle.
Being a fold does not depend on that choice, two lifts differing by a constant, and by
\cref{eq:framerelations} the condition $\varphi'(s)=0$ holds if and only if $\cos\psi(s)=0$, equivalently
$\abs{h'(s)}=1$, equivalently $\sweepdens(s)=0$ for the sweep density of \cref{eq:sweepdef}: at a fold the motion
is purely radial.
The name notwithstanding, $\varphi'$ need not change sign at a fold, and the sense of travel need not reverse;
\cref{sec:heuristics} uses ``fold'' in the narrower informal sense of a turning point.
Folds are not the same notion as the $\tfrac12$-alternations counted in \cref{prop:scarcity}, and the two are
kept apart.

Two points about the lift are worth stressing.
It is taken on the interval $[0,L]$ and not on $\mathbb{R}/L\mathbb{Z}$: it is periodic only if
$\ind_{\gamma}(0)=0$, and in general $\varphi(L)-\varphi(0)=2\pi\ind_{\gamma}(0)$.
We extend it to the whole of $\mathbb{R}$, the universal cover of the parameter circle, by
\[
  \varphi(s+L):=\varphi(s)+2\pi\ind_{\gamma}(0),
\]
which is again continuous and again lifts $\gamma/r$, the curve being $L$-periodic; parameters below range
over $\mathbb{R}$ whenever a construction is cyclic --- so for the endpoints $s_{\pm}$ of
\cref{sec:lower-head}, for the intervals $P_{j}$ of \cref{lem:separation}, and for the arcs of
$\gamma\setminus F_{\epsilon}$ in \cref{lem:blocking} --- and the corresponding subsets of the curve are their
images in $\mathbb{R}/L\mathbb{Z}$.
And it is the lift of $\gamma$ itself, not of the tangent: the tangent angle is a different function, related
to $\varphi$ through \cref{eq:frame}.

\begin{lemma}
  \label{lem:lift}
  Let $\gamma\in W^{2,2}(\mathbb{R}/L\mathbb{Z};\mathbb{R}^{2})$ be closed and parametrised by arclength, with
  $0\notin\gamma([0,L])$, and write $r:=\abs\gamma$.
  Then $\gamma$ has an angular lift, any two angular lifts differ by a constant in $2\pi\mathbb{Z}$, and every
  one of them belongs to $W^{2,2}(0,L)\subset C^{1}([0,L])$ and satisfies
  \begin{equation}
    \label{eq:liftderivative}
    \varphi'=\frac{\det(\gamma,\gamma')}{r^{2}},
    \qquad
    (r')^{2}+r^{2}(\varphi')^{2}\equiv 1 .
  \end{equation}
  Moreover $\int_{0}^{L}\varphi'\,ds=2\pi\ind_{\gamma}(0)$, and at every $s$ with $\varphi'(s)=0$ the tangent
  is parallel to $\gamma(s)$.
\end{lemma}

\begin{proof}
  The map $\sigma:=\gamma/r$ is of class $C^{1}$ into $S^{1}$, because $\gamma\in C^{1}$ and
  $r\ge\dist(0,\gamma([0,L]))>0$.
  The exponential $t\mapsto e^{it}$ is a covering $\mathbb{R}\to S^{1}$ and $[0,L]$ is simply connected, so
  $\sigma$ lifts: there is a continuous $\varphi$ with $\sigma=e^{i\varphi}$, unique once its value at $0$ is
  chosen, hence unique up to $2\pi\mathbb{Z}$ \cite[Prop.~1.33 and 1.34]{Hatcher2002}.
  Only the continuity of $\sigma$ is used here; the lifting of maps of low regularity is a delicate matter, for
  which we refer to \cite{BourgainBrezisMironescu2000}, but no such theory is needed, the regularity of $\varphi$
  following from the formula below.
  Differentiating $e^{i\varphi}=\gamma/r$ and pairing with $i\sigma$ gives the first identity in
  \cref{eq:liftderivative}, the contribution of $r'$ vanishing because $\langle a,ia\rangle=0$.
  The right-hand side is continuous and lies in $W^{1,2}$, whence $\varphi\in W^{2,2}$.
  Writing $\gamma=re^{i\varphi}$ and differentiating gives $\gamma'=(r'+ir\varphi')e^{i\varphi}$, and taking
  moduli yields the second identity in \cref{eq:liftderivative}; in particular $\varphi'(s)=0$ forces
  $\abs{r'(s)}=1$ and $\gamma'(s)=\pm\gamma(s)/r(s)$.
  Finally $\int_{0}^{L}\varphi'$ is $2\pi$ times the degree of $\sigma$, which is $\ind_{\gamma}(0)$.
\end{proof}

Since $r=1-h$, the second identity of \cref{eq:liftderivative} is \cref{eq:framerelations}; and the last
assertion is the precise sense in which a fold is a point of purely radial motion, used repeatedly in
\cref{sec:lower-envelope,sec:lower-head}.

\subsection{A master identity, and the energy budgets}
\label{sec:budgets}

All the estimates below issue from one identity, namely that
$\frac{d}{ds}\langle\gamma,\tau\rangle=1+\kappa\langle\gamma,\nu\rangle$ integrates to zero over a closed curve.
What makes them useful is that they have \emph{pointwise nonnegative} densities.
The density $\kappa^{2}-1$ of the excess is itself sign-indefinite, so the excess does not localise directly;
it does localise, however, up to a single boundary term, and we record that refinement first, in two forms,
since it is what allows an arc of the curve to be priced in isolation.

For an arc $J=[a,b]$ we write
\begin{equation}
  \label{eq:arcexcess}
  \excess_{J}:=\int_{J}\bigl(\kappa^{2}-1\bigr)ds,
  \qquad
  \bterm_{J}:=\bigl[\langle\gamma,\tau\rangle\bigr]_{a}^{b},
  \qquad
  \widehat{\excess}_{J}:=\excess_{J}+2\bterm_{J},
\end{equation}
and $\ell_{J}:=\abs{J}$, $\energy_{J}:=\int_{J}\kappa^{2}$, $\turning_{J}:=\int_{J}\abs{\kappa}$.
Since $\langle\gamma,\tau\rangle=\tfrac12\bigl(\abs{\gamma}^{2}\bigr)'$ is defined everywhere, and equals
$-(1-h)h'$ wherever $r>0$, the boundary term obeys
\begin{equation}
  \label{eq:boundaryterm}
  \abs{\bterm_{J}}\;\le\;\abs{h'(a)}+\abs{h'(b)} ,
\end{equation}
so it is small at any pair of cut points at which the curve runs almost parallel to $\partial\disk$, and it
vanishes at any pair of critical points of the depth.

\begin{proposition}
  \label{prop:masteridentity}
  For every arc $J$ of a closed curve $\gamma\in W^{2,2}$ parametrised by arclength with
  $\gamma\subset\overline{\disk}$,
  \begin{align}
    \widehat{\excess}_{J}
    &\;=\;\int_{J}\bigl(\abs{\kappa}-1\bigr)^{2}ds
    \;+\;2\int_{J}\abs{\kappa}\Bigl(1+\sgn(\kappa)\langle\gamma,\nu\rangle\Bigr)ds
    \label{eq:master}\\
    &\;=\;\int_{J}\bigl(1-r^{2}\bigr)ds\;+\;\int_{J}\langle\gamma,\tau\rangle^{2}ds
    \;+\;\int_{J}\bigl(\kappa-\sweepdens\bigr)^{2}ds ,
    \label{eq:master2}
  \end{align}
  with $\sweepdens=\det(\gamma,\tau)=-\langle\gamma,\nu\rangle$ as in \cref{eq:sweepdef}.
  The two right-hand sides are two splittings of the same density
  \begin{equation}
    \label{eq:masterdensity}
    \mathcal{D}\;:=\;1+\kappa^{2}+2\kappa\langle\gamma,\nu\rangle\;=\;1+\kappa^{2}-2\kappa\sweepdens ,
  \end{equation}
  and all five integrands are nonnegative pointwise; in particular $\widehat{\excess}_{J}\ge0$.
\end{proposition}

\begin{proof}
  Integrating $\frac{d}{ds}\langle\gamma,\tau\rangle=1+\kappa\langle\gamma,\nu\rangle$ over $J$ gives
  \begin{equation}
    \label{eq:closednessarc}
    \bterm_{J}=\ell_{J}+\int_{J}\kappa\langle\gamma,\nu\rangle\,ds ,
  \end{equation}
  so that $\widehat{\excess}_{J}=\energy_{J}-\ell_{J}+2\bterm_{J}=\int_{J}\bigl(\kappa^{2}+1+2\kappa\langle\gamma,\nu\rangle\bigr)ds
  =\int_{J}\mathcal{D}\,ds$.
  Expanding $(\abs{\kappa}-1)^{2}+2\abs{\kappa}(1+\sgn(\kappa)\langle\gamma,\nu\rangle)
  =\kappa^{2}-2\abs{\kappa}+1+2\abs{\kappa}+2\kappa\langle\gamma,\nu\rangle$ gives \cref{eq:master}.
  For \cref{eq:master2}, decompose $\gamma=\langle\gamma,\tau\rangle\tau+\langle\gamma,\nu\rangle\nu$, so that
  $\langle\gamma,\tau\rangle^{2}+\sweepdens^{2}=r^{2}$; then
  $(1-r^{2})+\langle\gamma,\tau\rangle^{2}+(\kappa-\sweepdens)^{2}
  =1-r^{2}+r^{2}+\kappa^{2}-2\kappa\sweepdens=\mathcal{D}$.
  The second integrand of \cref{eq:master} is nonnegative because $\abs{\langle\gamma,\nu\rangle}\le\abs{\gamma}\le1$,
  and the three integrands of \cref{eq:master2} visibly so.
\end{proof}

Taking $J$ to be the whole circle, where $\bterm_{J}=0$, recovers the four estimates of \cref{prop:budgets}
below simultaneously: \cref{eq:master} is the common ancestor of \cref{eq:B3,eq:B4}, and \cref{eq:master2}
that of \cref{eq:B1,eq:B2}, in each case with the boundary term made explicit.
The two splittings read the same density in complementary ways.
In \cref{eq:master2} the three terms measure, separately, how deep the curve is, how tilted it is, and how far
it is from satisfying the equation $\kappa=\sweepdens$ of the unit circle, which holds on $\partial\disk$
traversed in either sense; \cref{eq:master} instead separates the size of the curvature from its sign, and the
curvature-weighted term is what sees an arc of $\partial\disk$ traversed against its own curvature
(\cref{rem:budgetreading}).
The sharp assembly of \cref{sec:lower-sharp} uses \cref{eq:master2} only.

\begin{remark}
  \label{rem:B1known}
  \Cref{eq:master2} for the whole curve is not new: \citet[Lem.~2.4]{Wojtowytsch2018} proves, in the form
  $\energy(\gamma)=2L-\int_{0}^{L}r^{2}\,ds+\int_{0}^{L}\abs{\gamma''+\gamma}^{2}\,ds$, the identity
  $\excess=\int_{0}^{L}(1-r^{2})\,ds+\int_{0}^{L}\abs{\gamma''+\gamma}^{2}\,ds$, of which
  $\abs{\gamma''+\gamma}^{2}=\langle\gamma,\tau\rangle^{2}+(\kappa-\sweepdens)^{2}$ is the frame decomposition;
  the surface version is \citet[Thm.~1]{MullerRoger2014}.
  What is used below is that it is a special case of \cref{eq:master}, and therefore localises with the same
  boundary term.
\end{remark}

\begin{proposition}
  \label{prop:budgets}
  Let $\gamma$ be a closed curve of class $W^{2,2}$ parametrised by arclength, of length $L$, with
  $\gamma([0,L])\subset\overline{\disk}$.
  Then $\turning\ge L$ and
  \begin{align}
    \int_{0}^{L}\bigl(1-r^{2}\bigr)\,ds &\;\le\;\excess, \label{eq:B1}\\
    \int_{0}^{L}\langle\gamma,\tau\rangle^{2}\,ds &\;\le\;\excess, \label{eq:B2}\\
    \int_{0}^{L}\bigl(\abs{\kappa}-1\bigr)^{2}\,ds &\;=\;\excess-2\bigl(\turning-L\bigr)\;\le\;\excess,
      \label{eq:B3}\\
    \int_{0}^{L}\abs{\kappa}\Bigl(1+\sgn(\kappa)\,\langle\gamma,\nu\rangle\Bigr)\,ds
      &\;=\;\turning-L\;\le\;\tfrac12\excess. \label{eq:B4}
  \end{align}
  The integrands in \cref{eq:B1,eq:B2,eq:B4} are nonnegative pointwise, as is that of \cref{eq:B3}.
\end{proposition}

\begin{proof}
  Closedness gives
  \begin{equation}
    \label{eq:closedness}
    0=\oint\frac{d}{ds}\langle\gamma,\tau\rangle\,ds=L+\oint\kappa\,\langle\gamma,\nu\rangle\,ds ,
  \end{equation}
  and since $\abs{\langle\gamma,\nu\rangle}\le\abs{\gamma}\le1$ this yields $L\le\oint\abs{\kappa}=\turning$,
  which is Chakerian's inequality \citep{Chakerian1962}.
  For the whole curve $\bterm_{[0,L]}=0$, so $\widehat{\excess}_{[0,L]}=\excess$, and \cref{eq:B1,eq:B2} are
  \cref{eq:master2} with two of its three nonnegative terms discarded.
  For \cref{eq:B3} expand:
  $\int(\abs{\kappa}-1)^{2}=\energy-2\turning+L=\excess-2(\turning-L)$, and $\turning\ge L$ gives the inequality.
  For \cref{eq:B4}, add $\turning$ to \cref{eq:closedness}:
  \[
    \turning-L=\oint\abs{\kappa}\,ds+\oint\kappa\langle\gamma,\nu\rangle\,ds
    =\oint\abs{\kappa}\bigl(1+\sgn(\kappa)\langle\gamma,\nu\rangle\bigr)ds ,
  \]
  whose integrand is nonnegative because $\abs{\langle\gamma,\nu\rangle}\le1$; the bound
  $\turning-L\le\tfrac12\excess$ is \cref{eq:B3}.
\end{proof}

All four estimates hold on an arbitrary arc, with $\excess$ replaced by the charge $\widehat{\excess}_{J}$ of
\cref{eq:arcexcess}: this is \cref{cor:localbudgets}, which is simply \cref{prop:masteridentity} with all but
one term discarded, and which is what makes it possible to price one part of the curve without reference to
the rest, provided only that the cut points are chosen where the curve is almost tangential, so that
\cref{eq:boundaryterm} is small.
Moreover the charge is \emph{superadditive}: the density $\mathcal{D}$ of \cref{eq:masterdensity} is
pointwise nonnegative and $\widehat{\excess}_{J}=\int_{J}\mathcal{D}$, so for pairwise disjoint arcs
$J_{1},\dots,J_{m}$,
\begin{equation}
  \label{eq:additivity}
  \sum_{i=1}^{m}\widehat{\excess}_{J_{i}}\;\le\;\int_{0}^{L}\mathcal{D}\,ds\;=\;\widehat{\excess}_{[0,L]}
  \;=\;\excess .
\end{equation}
This inequality is how the fold price and the dip prices will be added in \cref{sec:lower-sharp}.

\begin{corollary}
  \label{cor:localbudgets}
  For every arc $J$,
  \begin{align}
    \int_{J}\bigl(1-r^{2}\bigr)ds &\;\le\;\widehat{\excess}_{J}, \label{eq:B1loc}\\
    \int_{J}\langle\gamma,\tau\rangle^{2}ds &\;\le\;\widehat{\excess}_{J}, \label{eq:B2loc}\\
    \int_{J}\bigl(\kappa-\sweepdens\bigr)^{2}ds &\;\le\;\widehat{\excess}_{J}, \label{eq:B5loc}\\
    \int_{J}\bigl(\abs{\kappa}-1\bigr)^{2}ds &\;\le\;\widehat{\excess}_{J}, \label{eq:B3loc}\\
    \int_{J}\abs{\kappa}\Bigl(1+\sgn(\kappa)\langle\gamma,\nu\rangle\Bigr)ds
      &\;\le\;\tfrac12\widehat{\excess}_{J}. \label{eq:B4loc}
  \end{align}
  In particular, by \cref{eq:depthbound}, $\int_{J}h\,ds\le\widehat{\excess}_{J}$ for every arc $J$.
\end{corollary}

\begin{proof}
  Each is one of \cref{eq:master,eq:master2} with the other terms discarded.
\end{proof}

\begin{corollary}
  \label{cor:hprime}
  $\displaystyle\int_{0}^{L}(h')^{2}\,ds\ \le\ 6\,\excess$.
\end{corollary}

\begin{proof}
  On $\{h\le\tfrac12\}$ one has $\langle\gamma,\tau\rangle=-(1-h)h'$ and $1-h\ge\tfrac12$, so
  $(h')^{2}\le4\langle\gamma,\tau\rangle^{2}$ there, and the integral over this set is at most $4\excess$
  by \cref{eq:B2}.
  On $\{h>\tfrac12\}$ the Lipschitz bound gives $(h')^{2}\le1$, while
  $\abs{\{h>\tfrac12\}}\le2\int h\le2\excess$ by \cref{eq:B1,eq:depthbound}.
\end{proof}

\begin{remark}
  \label{rem:tangentialcut}
  It is worth recording when the boundary term is free.
  An arc cut at two critical points of the depth, or one whose two endpoints lie on $\partial\disk$
  tangentially, has $\bterm_{J}=0$, and therefore $\excess_{J}=\widehat{\excess}_{J}\ge0$: for such arcs the
  inequality $\energy\ge L$ of \cref{eq:trivialbound} persists, with equality only for arcs of
  $\partial\disk$.
  The cuts made below are of neither kind --- they are made where $\abs{h'}=\sin\epsilon$, and
  \cref{eq:boundaryterm} is used instead --- but the two cases delimit how small the boundary term can be
  made.
\end{remark}

\begin{remark}
  \label{rem:budgetreading}
  The four estimates say complementary things.
  \Cref{eq:B1,eq:B2} state that a curve of small excess is \emph{positioned} like a multiply covered unit circle:
  outside a set of length proportional to $\excess$ it is close to $\partial\disk$ and almost tangential.
  \Cref{eq:B3} states that it is also \emph{curved} like one, $\abs{\kappa}$ being close to $1$ in $L^{2}$.
  \Cref{eq:B4} is the only one weighted by curvature.
  By \cref{eq:sweepdef} its density reads
  \[
    \abs{\kappa}\Bigl(1-\sgn(\kappa)\,r\cos\psi\Bigr),
  \]
  which vanishes precisely where $\kappa=0$ or where simultaneously $r=1$ and $\cos\psi=\sgn(\kappa)$, that is
  on straight pieces and on the unit circle traversed in the sense of its own curvature; away from straight
  pieces it charges simultaneously, weighted by $\abs{\kappa}$, the three ways of deviating from the circle:
  curving against the sense of travel ($\sgn(\kappa)\cos\psi<0$), being deep ($r<1$), and being tilted
  ($\abs{\cos\psi}<1$).
  It is the sign information that \cref{eq:B1,eq:B2,eq:B3} lack: an arc of unit curvature hugging
  $\partial\disk$ but curving away from the origin is shallow, tangential, and has $\abs{\kappa}=1$, so the three
  do not see it at all, whereas the density of \cref{eq:B4} equals $2$ there.
  In the second splitting \cref{eq:master2} the same arc is charged by $(\kappa-\sweepdens)^{2}=4$: the term
  $(\kappa-\sweepdens)^{2}$ carries the sign of the curvature together with its size, which is why
  \cref{sec:lower-dip,sec:lower-sharp} can work with \cref{eq:master2} alone and never separate the two.
  \Cref{eq:B4} itself is not used in the proofs below; it is recorded because it is the half of \cref{eq:master}
  that makes the sign visible.
\end{remark}

\subsection{Alternations are few}
\label{sec:lower-scarcity}

The natural variable here is the sweep density $\sweepdens=\det(\gamma,\tau)=r\cos\psi$ of
\cref{eq:sweepdef} rather than $\varphi'$: the two have the same sign, but $\varphi'=\sweepdens/r^{2}$ carries
a factor which blows up where the curve passes close to the origin, whereas the sweep density does not.
Accordingly, nothing in this subsection requires the competitor to avoid the origin: the counting arguments
use only $\sweepdens$ itself, which is defined everywhere.

\begin{lemma}
  \label{lem:omega}
  Let $\gamma\in\emb_{L}$.
  Then $\sweepdens\in W^{1,2}(\mathbb{R}/L\mathbb{Z})$ and
  \[
    \abs{\sweepdens}\le r\le1,\qquad
    \sweepdens'=\kappa\,\langle\gamma,\tau\rangle,\qquad
    \abs{\sweepdens'}\le\abs{\kappa},\qquad
    \sweepdens^{2}+\langle\gamma,\tau\rangle^{2}=r^{2}.
  \]
  If moreover $0\notin\gamma([0,L])$, then $\sweepdens=r^{2}\varphi'$.
\end{lemma}

\begin{proof}
  Decomposing $\tau$ along $\gamma/r$ and its rotation by $\tfrac\pi2$ gives the Pythagorean identity
  and $\abs{\sweepdens}\le r$ wherever $r>0$; at a parameter with $\gamma(s)=0$ both sides vanish.
  Differentiating, $\sweepdens'=\det(\tau,\tau)+\det(\gamma,\kappa\nu)=\kappa\det(\gamma,i\tau)
  =\kappa\langle\gamma,\tau\rangle$, and $\abs{\langle\gamma,\tau\rangle}\le r\le1$.
  The last identity is \cref{eq:liftderivative}.
\end{proof}

Folds are the zeros of $\sweepdens$, and sign changes of $\sweepdens$ are folds at which the sense of travel
does reverse.
Neither number can be bounded: if $\sweepdens(s)=s^{3}\sin(1/s)$ near a point of purely radial motion, then
$\sweepdens$ changes sign infinitely often while, by \cref{lem:omega}, the curvature needed is
$\abs{\kappa}=\abs{\sweepdens'}/\abs{\langle\gamma,\tau\rangle}=O(\abs{s})$, so the configuration is
admissible and cheap.
What is bounded, and what the argument needs, is the number of sign changes at a definite amplitude.

\begin{definition}
  \label{def:alternations}
  For $\varepsilon>0$, an \emph{$\varepsilon$-alternating chain} is a finite cyclic sequence
  $s_{1}<\dots<s_{n}$ in $[0,L)$ with $\abs{\sweepdens(s_{j})}\ge\varepsilon$ for every $j$ and with
  $\sgn\sweepdens(s_{j})$ alternating cyclically.
  We write $N_{\varepsilon}$ for the supremum of $n$ over all such chains.
  Indices are understood modulo $n$ and the intervals $[s_{j},s_{j+1}]$ in $\mathbb{R}/L\mathbb{Z}$, so that
  $j=n$ denotes the wrap-around arc.
  Every alternating chain has even length, so $N_{\varepsilon}\ne1$; that $N_{\varepsilon}$ is finite is the
  content of \cref{prop:scarcity}.
\end{definition}

The following bound is unconditional: it requires neither the standing assumption \cref{eq:avoidorigin} nor
any smallness hypothesis on $\excess$.

\begin{proposition}
  \label{prop:scarcity}
  Let $\gamma\in\emb_{L}$.
  Then $N_{1/2}\le C\,\excess$.
\end{proposition}

\begin{proof}
  Assume $N_{1/2}\ge2$ and let $s_{1}<\dots<s_{n}$ be a $\tfrac12$-alternating chain with
  $\sigma_{j}:=\sgn\sweepdens(s_{j})$.
  For each $j$ put
  \begin{align*}
    a_{j}&:=\max\{s\in[s_{j},s_{j+1}]:\sigma_{j}\sweepdens(s)\ge\tfrac12\},\\
    b_{j}&:=\min\{s\in[a_{j},s_{j+1}]:\sigma_{j}\sweepdens(s)\le-\tfrac12\},
    \qquad I_{j}:=[a_{j},b_{j}].
  \end{align*}
  Both sets are nonempty and closed, so $a_{j}$ and $b_{j}$ are well defined.
  Moreover $a_{j}<s_{j+1}$, because $\sigma_{j}\sweepdens(s_{j+1})\le-\tfrac12$, so by continuity
  $\sigma_{j}\sweepdens(a_{j})=\tfrac12$; symmetrically $\sigma_{j}\sweepdens(b_{j})=-\tfrac12$, and
  $\abs{\sweepdens}\le\tfrac12$ on $I_{j}$ by maximality of $a_{j}$ and minimality of $b_{j}$.
  Since $\sweepdens\in W^{1,2}$ is absolutely continuous and $\abs{\sigma_{j}}=1$,
  \[
    \int_{I_{j}}\abs{\sweepdens'}\,ds\ \ge\ \Bigl|\int_{I_{j}}\sweepdens'\,ds\Bigr|
    =\bigl|\sweepdens(b_{j})-\sweepdens(a_{j})\bigr|=1 :
  \]
  the sweep density must cross the whole band $[-\tfrac12,\tfrac12]$ on $I_{j}$.
  Since $I_{j}\subset[s_{j},s_{j+1}]$ the intervals have pairwise disjoint interiors, and by \cref{lem:omega}
  \begin{equation}
    \label{eq:transitionturning}
    \int_{I_{j}}\abs{\kappa}\,ds\ \ge\ \int_{I_{j}}\abs{\sweepdens'}\,ds\ \ge\ 1 .
  \end{equation}
  Write $\ell_{j}:=\abs{I_{j}}$ and split the indices according to whether $\ell_{j}$ exceeds a threshold
  $\ell_{*}\in(0,\tfrac12)$ to be fixed.

  \emph{Long intervals.}
  Let $s\in I_{j}$ with $r(s)\ge\tfrac9{10}$.
  Since $\abs{\sweepdens}\le\tfrac12$ on $I_{j}$, the Pythagorean identity of \cref{lem:omega} gives
  $\abs{\langle\gamma,\tau\rangle}=\bigl(r^{2}-\sweepdens^{2}\bigr)^{1/2}
  \ge\bigl(\tfrac{81}{100}-\tfrac14\bigr)^{1/2}\ge\tfrac12$, whence
  $\langle\gamma,\tau\rangle^{2}\ge\tfrac14$ there.
  As the $I_{j}$ are disjoint, \cref{eq:B2} gives $\sum_{j}\abs{I_{j}\cap\{r\ge\tfrac9{10}\}}\le4\excess$, while
  \cref{eq:B1} gives $\abs{\{r<\tfrac9{10}\}}\le C\excess$.
  Hence $\sum_{j}\ell_{j}\le C\excess$ and the number of long intervals is at most $C\excess/\ell_{*}$.

  \emph{Short intervals.}
  If $\ell_{j}\le\ell_{*}$ then by \cref{eq:transitionturning}
  \[
    \int_{I_{j}}\bigl|\abs{\kappa}-1\bigr|ds\ \ge\ \int_{I_{j}}\abs{\kappa}\,ds-\ell_{j}\ \ge\ 1-\ell_{*}
    \ \ge\ \tfrac12,
  \]
  so Cauchy--Schwarz on $I_{j}$ gives
  $\int_{I_{j}}(\abs{\kappa}-1)^{2}ds\ge1/(4\ell_{j})\ge1/(4\ell_{*})$.
  Summing over the disjoint short intervals and using \cref{eq:B3}, their number is at most
  $C\ell_{*}\excess$.

  Fixing $\ell_{*}$ once and for all, both counts are at most $C\excess$, and the bound is independent of the
  chain.
\end{proof}

\begin{corollary}
  \label{cor:arcs}
  Let $\gamma\in\emb_{L}$ and write $Z:=\{s:\abs{\sweepdens(s)}<\tfrac12\}$.
  \begin{enumerate}
    \item There is a partition of $[0,L)$ into $\max(N_{1/2},1)\le C\max(\excess,1)$ arcs
      $\Pi_{1},\Pi_{2},\dots$ such
      that $\sgn\sweepdens$ is constant on each $\Pi_{k}\setminus Z$; moreover $\abs{Z}\le C\excess$.
    \item If \cref{eq:avoidorigin} holds, then in addition
      \[
        \abs{\varphi'}=\frac{\abs{\sweepdens}}{r^{2}}\ \le\ 2
        \qquad\text{on }Z\cap\bigl\{h\le\tfrac12\bigr\}.
      \]
  \end{enumerate}
\end{corollary}

\begin{proof}
  (i) If $N_{1/2}=0$, then $\sweepdens$ omits one of the two values $\pm\tfrac12$ --- two parameters
  attaining them with opposite signs would form an alternating chain of length $2$ --- so $\sgn\sweepdens$ is
  constant off $Z$ on the single arc $[0,L)$, and only the bound on $\abs Z$, proved below without reference
  to any chain, remains; since $N_{1/2}$ is even, assume then $N_{1/2}\ge2$.
  Fix a chain realising $N_{1/2}$, which is attained because $N_{1/2}$ is a finite integer by
  \cref{prop:scarcity}, and let $I_{j}=[a_{j},b_{j}]$ be its transition intervals, as in the proof of that
  theorem.
  Set $\Pi_{j}:=[b_{j},b_{j+1})$, indices modulo $N_{1/2}$.
  These partition $[0,L)$: each absorbs the gap $[b_{j},a_{j+1}]$ together with the following transition
  interval $[a_{j+1},b_{j+1})$.
  Suppose now some $t\in\Pi_{j}$ has $\abs{\sweepdens(t)}\ge\tfrac12$ and $\sgn\sweepdens(t)=\sigma_{j}$, and
  recall $b_{j}\le s_{j+1}\le a_{j+1}$.
  If $t\le s_{j+1}$ then $\sigma_{j}\sweepdens(t)\ge\tfrac12$ with $a_{j}<t\le s_{j+1}$, contradicting the
  maximality of $a_{j}$.
  If $t>s_{j+1}$ then $s_{j+1}<t<a_{j+1}<s_{j+2}$: indeed $t<b_{j+1}$, and $a_{j+1}\le t$ would put $t$ in
  $[a_{j+1},s_{j+2}]$ with $\sigma_{j+1}\sweepdens(t)\le-\tfrac12$, contradicting the minimality of
  $b_{j+1}$.
  Inserting the \emph{pair} $t,a_{j+1}$ then produces an
  alternating chain of length $N_{1/2}+2$, contradicting maximality; a single insertion would not do, since an
  alternating cyclic chain has even length.
  Hence $\sgn\sweepdens$ is constant off $Z$ on each $\Pi_{j}$, equal to $-\sigma_{j}$.
  The bound on $\abs{Z}$ is the long-interval estimate of the proof of \cref{prop:scarcity}, applied to $Z$ in
  place of $\bigcup_{j}I_{j}$: on $Z\cap\{r\ge\tfrac9{10}\}$ one has $\langle\gamma,\tau\rangle^{2}\ge\tfrac14$,
  so \cref{eq:B2} bounds its measure by $4\excess$, while $\abs{\{r<\tfrac9{10}\}}\le C\excess$ by \cref{eq:B1}.

  (ii) On $Z\cap\{h\le\tfrac12\}$ one has $r\ge\tfrac12$ and $\abs{\sweepdens}<\tfrac12$, so
  $\abs{\varphi'}=\abs{\sweepdens}/r^{2}<2$ by \cref{eq:framerelations}; under \cref{eq:avoidorigin} the
  quantity is defined everywhere there.
\end{proof}

\subsection{A fold on the outer envelope}
\label{sec:lower-envelope}

\begin{definition}
  \label{def:envelope}
  Let $\Theta$ be the set of angles $\theta$ whose open ray $\{te^{i\theta}:t>0\}$ meets $\gamma([0,L])$.
  The \emph{envelope radius} is
  \[
    R(\theta):=\max\bigl\{t>0:\ te^{i\theta}\in\gamma([0,L])\bigr\},\qquad\theta\in\Theta,
  \]
  the \emph{outer envelope} is the set $\mathrm{Env}:=\{R(\theta)e^{i\theta}:\theta\in\Theta\}$ of the
  corresponding points, and a parameter $s$ is an \emph{envelope parameter} if $\gamma(s)\in\mathrm{Env}$,
  that is, if $\gamma(s)$ is the outermost point of the trace on its own ray.
\end{definition}

The maximum defining $R$ is attained, the trace being compact and, by \cref{eq:avoidorigin}, bounded away from
the origin; so every $\theta\in\Theta$ carries at least one envelope parameter.
About the origin, being outermost on a ray is the same as being shallowest on it, since $h(te^{i\theta})=1-t$;
no analogue of this holds about any other base point, and it is what makes the origin the right centre for the
envelope.

\begin{lemma}
  \label{lem:env}
  For $L\ge L_{0}(\bootstrap)$ there is a fold $s^{*}$ with $\gamma(s^{*})\in\mathrm{Env}$.
\end{lemma}

\begin{proof}
  By \cref{eq:avoidorigin} the map $s\mapsto\gamma(s)/\abs{\gamma(s)}$ is continuous, so the set $\Theta$ of
  \cref{def:envelope}, being its image, is compact and connected.

  \emph{Case 1: some ray misses the trace.}
  Then $\Theta$ is a proper compact connected subset of $S^{1}$, that is, a closed arc, possibly degenerate;
  let $\theta^{*}$ be an endpoint of that arc, let $q$ be the outermost point of the trace on the ray of
  $\theta^{*}$, and let $s_{q}$ be a parameter with $\gamma(s_{q})=q$.
  There is $\rho>0$ such that $\Theta$ meets $(\theta^{*}-\rho,\theta^{*}+\rho)$ on at most one side of
  $\theta^{*}$.
  For $s$ near $s_{q}$ the value $\varphi(s)$ lies in $\Theta$ modulo $2\pi$ and is close to
  $\varphi(s_{q})\equiv\theta^{*}$; hence the lift $\varphi$ has a local extremum at $s_{q}$, so
  $\varphi'(s_{q})=0$: the parameter $s_{q}$ is a fold, and $q\in\mathrm{Env}$ by the choice of $q$.

  \emph{Case 2: every ray meets the trace.}
  Then $R$ is defined on all of $S^{1}$, with $R\ge\min\abs\gamma>0$, and it is upper semicontinuous: if
  $\theta_{n}\to\theta$ and $R(\theta_{n})\to\ell$, then $\ell e^{i\theta}\in\gamma([0,L])$ by compactness,
  and $\ell\ge\min\abs\gamma>0$, so $\ell\le R(\theta)$.
  Suppose, for contradiction, that no envelope parameter is a fold.
  Fix $\theta_{1}$ and let $q:=R(\theta_{1})e^{i\theta_{1}}$.
  Since $\gamma$ is injective there is exactly one parameter over $q$; it is an envelope parameter, so
  $\varphi'\ne0$ there, and near that parameter the curve is a polar graph $t=\varrho(\theta)$ about the
  origin over a neighbourhood of $\theta_{1}$, with $\varrho$ of class $C^{1}$ and $R\ge\varrho$.
  As $R$ is upper semicontinuous and bounded below near $\theta_{1}$ by the continuous function $\varrho$
  with $R(\theta_{1})=\varrho(\theta_{1})$, it is continuous at $\theta_{1}$.

  It also coincides with $\varrho$ near $\theta_{1}$, and here injectivity is essential: continuity alone does
  not suffice, as $R=\varrho+\dist(\cdot,\theta_{1})$ shows.
  Write $I_{0}$ for the parameter interval on which $\gamma$ is the graph $\varrho$ and $K_{0}$ for its
  complement.
  Then $\gamma(K_{0})$ is compact and, $\gamma$ being injective and $q$ the image of a parameter of $I_{0}$,
  we have $q\notin\gamma(K_{0})$, so $d:=\dist(q,\gamma(K_{0}))>0$.
  By the continuity just proved, the outermost point $p(\theta)=R(\theta)e^{i\theta}$ satisfies
  $\abs{p(\theta)-q}<d$ for $\theta$ near $\theta_{1}$, hence $p(\theta)\in\gamma(I_{0})$; lying on the ray of
  angle $\theta$, it has radius $\varrho(\theta)$.
  So $R=\varrho$ near $\theta_{1}$, and $\mathrm{Env}$ is a closed curve of class $C^{1}$ contained in
  $\gamma([0,L])$.

  A closed connected subset of a Jordan curve is a point, a closed arc, or the whole curve, none of the first
  two being a closed curve; hence $\mathrm{Env}=\gamma([0,L])$.
  Every parameter is therefore an envelope parameter, and by the contradiction hypothesis $\varphi'$ vanishes
  nowhere; being continuous, it has constant sign, so the lift $\varphi$ is strictly monotone and its total
  variation over $[0,L]$ equals $\abs{\smash[b]{\int_{0}^{L}\varphi'}}=2\pi\ind_{\gamma}(0)\le2\pi$ by
  \cref{lem:lift}.
  On the other hand, on $G:=\{h\le\tfrac1{10},\abs{h'}\le\tfrac1{10}\}$ one has, by \cref{eq:framerelations},
  $\abs{\varphi'}=\sqrt{1-(h')^{2}}/(1-h)\ge\sqrt{1-\tfrac1{100}}\ge\tfrac9{10}$, while
  $\abs{[0,L]\setminus G}\le10\int h+100\int(h')^{2}\le C\excess$ by \cref{eq:B1,eq:depthbound,cor:hprime}.
  Hence the total variation is at least $\tfrac9{10}(L-C\excess)>2\pi$ for $L\ge L_{0}(\bootstrap)$, by
  \cref{eq:bootstrap}: a contradiction.
\end{proof}

\begin{proposition}
  \label{prop:uncovered}
  Let $s^{*}$ be as in \cref{lem:env}, $q^{*}:=\gamma(s^{*})$, $D:=h(s^{*})$, $\theta_{0}:=\varphi(s^{*})$,
  and let $\hat q^{*}:=q^{*}/\abs{q^{*}}\in\partial\disk$ be the boundary point on the same ray.
  Then $D>0$, the half-open segment $(q^{*},\hat q^{*}]$ does not meet $\gamma([0,L])$, and every parameter
  $s$ with $\gamma(s)$ on the ray of $\theta_{0}$ satisfies $h(s)\ge D$.
\end{proposition}

\begin{proof}
  The depth is nonnegative and, by \cref{eq:avoidorigin}, of class $C^{1}$; so if $D$ vanished then $s^{*}$
  would be a minimum of $h$ and $h'(s^{*})=0$, whereas $\abs{h'(s^{*})}=1$ at a fold by \cref{def:lift}.
  Hence $D>0$.
  Points of $(q^{*},\hat q^{*}]$ lie on the ray of $\theta_{0}$ at radius exceeding
  $R(\theta_{0})=\abs{q^{*}}=1-D$, so none of them is on the trace, by the maximality defining $R$.
  For the last assertion, such a parameter has $\abs{\gamma(s)}\le R(\theta_{0})$, again by maximality.
  The first two assertions are used in \cref{lem:blocking}; the third records that nothing lies above an
  envelope fold, which is what makes such a fold worth selecting.
\end{proof}

\subsection{The cost of the fold}
\label{sec:lower-head}

The fold is expensive because at a fold the curve moves radially at unit speed, so that the confinement
$h\ge0$ forces it to turn within a length proportional to its distance to the boundary.
No strand above the fold is used; the barrier is $\partial\disk$ itself.
The turn is priced in the chart in which the depth is the independent variable, where the natural weight makes
the estimate an application of Cauchy--Schwarz and produces the sharp constant.

Throughout this subsection we write $\beta:=-\psi$, so that \cref{eq:framerelations} becomes
\begin{equation}
  \label{eq:betaframe}
  h'=\sin\beta,\qquad (1-h)\varphi'=\cos\beta,\qquad \kappa=\varphi'+\beta',
\end{equation}
and we set
\begin{equation}
  \label{eq:Ghat}
  \widehat G(\beta):=\int_{0}^{\beta}\sqrt{\abs{\sin t}}\,dt,
  \qquad
  \dgconst:=\widehat G(\pi)=\int_{0}^{\pi}\sqrt{\sin t}\,dt=\frac{4\sqrt2\,\pi^{3/2}}{\Gamma(1/4)^{2}} .
\end{equation}
Numerically $\dgconst=2.39628\ldots$.
The substitution $\tau=\cot\beta$ turns $\widehat G$ into the function
$G(s)=\int_{0}^{s}(1+\tau^{2})^{-5/4}d\tau$ of Deckelnick and Grunau \cite[Lem.~4]{DeckelnickGrunau2007}, through
$G(\cot\beta)=\widehat G(\tfrac\pi2)-\widehat G(\beta)$; in particular $\dgconst$ is the constant they denote
$c_{0}$, which is the \emph{two-sided} integral $\int_{\mathbb R}(1+\tau^{2})^{-5/4}d\tau=B(\tfrac12,\tfrac34)$,
so that their $G(+\infty)=c_{0}/2$ matches $\widehat G(\tfrac\pi2)=\dgconst/2$; and \cref{lem:depthchart} is
the form taken in the disk by the first integral $\kappa\,(1+u'^{2})^{1/4}\equiv\text{const}$ which
characterises their solutions.
In this identification the roles of the coordinates are exchanged: their independent variable corresponds to
the depth, so that the first integral reads $\kappa^{2}\propto\abs{\sin\beta}$ in our frame, and not
$\kappa^{2}\propto\abs{\cos\beta}$ as a naive reading of the graph picture would suggest.

\begin{lemma}
  \label{lem:depthchart}
  Let $J$ be an arc on which $h$ is strictly monotone, and let
  $\Delta h_{J}:=\int_{J}\abs{h'}\,ds$ be the range of the depth on $J$.
  Then
  \begin{equation}
    \label{eq:sweepidentity}
    \int_{J}\kappa\sqrt{\abs{\sin\beta}}\,ds
    \;=\;\bigl[\widehat G(\beta)\bigr]_{\partial J}
    \;+\;\int_{J}\frac{\sqrt{\abs{\sin\beta}}\,\cos\beta}{1-h}\,ds ,
  \end{equation}
  and, with $v:=\kappa\abs{\sin\beta}^{-1/2}$ and $\bar v:=\Delta h_{J}^{-1}\int_{J}\kappa\sqrt{\abs{\sin\beta}}\,ds$,
  \begin{equation}
    \label{eq:sweepCS}
    \energy_{J}
    \;=\;\frac{1}{\Delta h_{J}}\Bigl(\int_{J}\kappa\sqrt{\abs{\sin\beta}}\,ds\Bigr)^{2}
    \;+\;\int_{J}\bigl(v-\bar v\bigr)^{2}\abs{\sin\beta}\,ds .
  \end{equation}
\end{lemma}

\begin{proof}
  By \cref{eq:betaframe}, $\kappa\,ds=d\varphi+d\beta$ and $d\varphi=\frac{\cos\beta}{1-h}\,ds$, while
  $\int\sqrt{\abs{\sin\beta}}\,d\beta=\bigl[\widehat G(\beta)\bigr]$; this is \cref{eq:sweepidentity}.
  For \cref{eq:sweepCS}, monotonicity of $h$ allows the depth to be used as the variable of integration:
  $ds=\abs{\sin\beta}^{-1}\,d h$, so $\energy_{J}=\int v^{2}\abs{\sin\beta}\,ds=\int v^{2}\,dh$ and
  $\int v\,dh=\int_{J}\kappa\sqrt{\abs{\sin\beta}}\,ds$.
  Expanding $\int(v-\bar v)^{2}dh$ gives the identity.
\end{proof}

The second term in \cref{eq:sweepidentity} is the only place where the disk differs from a half-plane; it is
bounded by $\abs{J}/(1-\max_{J}h)$.
The second term in \cref{eq:sweepCS} is the defect in Cauchy--Schwarz, and is recorded because it is coercive:
an arc whose energy is close to the bound below has $\kappa^{2}$ close to a multiple of $\abs{\sin\beta}$, which
is the first integral of the free elastica.

\subsection{The cost of the passages beneath it}
\label{sec:lower-dip}

The passages beneath the fold are priced on disjoint parameter windows, one per passage; the windows are
produced in \cref{sec:lower-sharp} by the blocking argument, and on each of them the model functional
$\int(2h+(h'')^{2})$ is bounded below by a calibration inequality.
Here we record how the charge of an arc controls that functional.
The conversion rests on the second splitting \cref{eq:master2} of the master identity, whose first term is
$1-r^{2}=2h-h^{2}$, exactly the linear term of the model functional, and whose third term controls $(h'')^{2}$
through the following identity.

\begin{lemma}
  \label{lem:identity}
  On any arc contained in $\{h\le\tfrac1{10}\}$,
  \begin{equation}
    \label{eq:exactidentity}
    h''\;=\;\kappa\cos\psi-\frac{\cos^{2}\psi}{1-h}
    \;=\;\bigl(\kappa-\sweepdens\bigr)\cos\psi-e,
    \qquad
    e:=\cos^{2}\psi\,\frac{2h-h^{2}}{1-h}\in[0,3h] .
  \end{equation}
  Consequently $\abs{h''}\le\abs{\kappa-\sweepdens}+3h$, and for every $\lambda\in(0,1)$
  \begin{equation}
    \label{eq:hsecondbound}
    \bigl(\kappa-\sweepdens\bigr)^{2}\ \ge\ (1-\lambda)\,(h'')^{2}-\frac{9}{\lambda}\,h^{2} .
  \end{equation}
\end{lemma}

\begin{proof}
  Differentiating $h'=-\sin\psi$ gives $h''=-\psi'\cos\psi$, and substituting $\psi'=\varphi'-\kappa$ with
  $\varphi'=\cos\psi/(1-h)$ yields the first identity; note that no division by $\cos\psi$ occurs, so the identity
  is valid also where the motion is purely radial.
  Since $\sweepdens=r\cos\psi=(1-h)\cos\psi$ by \cref{eq:sweepdef}, and
  $\frac1{1-h}-(1-h)=\frac{2h-h^{2}}{1-h}$,
  \[
    \kappa\cos\psi-\frac{\cos^{2}\psi}{1-h}
    =(\kappa-\sweepdens)\cos\psi-\cos^{2}\psi\Bigl(\frac1{1-h}-(1-h)\Bigr)
    =(\kappa-\sweepdens)\cos\psi-e ,
  \]
  which is the second form, and $0\le e\le2h/(1-h)\le3h$ on $\{h\le\tfrac1{10}\}$.
  Hence $\abs{h''}\le\abs{\kappa-\sweepdens}+3h$, so that
  $(h'')^{2}\le(1+\lambda)(\kappa-\sweepdens)^{2}+9(1+\lambda^{-1})h^{2}$, and dividing by $1+\lambda$ gives
  \cref{eq:hsecondbound}, since $(1+\lambda)^{-1}\ge1-\lambda$ and $9(1+\lambda^{-1})/(1+\lambda)=9/\lambda$.
\end{proof}

\begin{proposition}
  \label{prop:conversion}
  Let $J$ be an arc contained in $\{h\le\tfrac1{10}\}$, let $H:=\max_{J}h$, and let $\lambda\in(0,1)$.
  Then
  \begin{equation}
    \label{eq:conversion}
    \widehat{\excess}_{J}\ \ge\ (1-\lambda)\int_{J}\Bigl(2h+(h'')^{2}\Bigr)ds\ -\ \frac{10}{\lambda}\,H\int_{J}h\,ds ,
    \qquad\text{and}\qquad
    \int_{J}h\,ds\ \le\ \widehat{\excess}_{J} .
  \end{equation}
\end{proposition}

\begin{proof}
  The second inequality is \cref{cor:localbudgets}.
  For the first, discard the middle term of \cref{eq:master2}, insert $1-r^{2}=2h-h^{2}$ and
  \cref{eq:hsecondbound}:
  \[
    \widehat{\excess}_{J}
    \ \ge\ \int_{J}\bigl(2h-h^{2}\bigr)+\int_{J}(\kappa-\sweepdens)^{2}
    \ \ge\ (1-\lambda)\int_{J}\Bigl(2h+(h'')^{2}\Bigr)-\Bigl(1+\frac9\lambda\Bigr)\int_{J}h^{2} ,
  \]
  and $\int_{J}h^{2}\le H\int_{J}h$ with $1+9/\lambda\le10/\lambda$.
\end{proof}

\begin{remark}
  \label{rem:conversionexact}
  Nothing in \cref{prop:conversion} depends on the sign of the curvature, and no hypothesis on $h'$ is needed.
  Both features come from \cref{eq:master2}: the term $(\kappa-\sweepdens)^{2}$ charges wrong-way curvature and
  right-way curvature alike, and the term $\langle\gamma,\tau\rangle^{2}=(1-h)^{2}(h')^{2}$, discarded above,
  is what a finer conversion would keep; see \cref{rem:whyLthird}.
\end{remark}

\subsection{The sharp constant}
\label{sec:lower-sharp}

The estimates above were arranged to produce the exponent, and the constants they carry are not optimal.
We now show that the same architecture, run with the sharp constant in each of its two model problems,
determines the coefficient of $L^{4/9}$.
Two constants enter, both already met: $\dgconst$ of \cref{eq:Ghat}, from the fold, and
\begin{equation}
  \label{eq:dipconst}
  \dipconst:=\frac{4\quartrad^{5}}{135}=\frac{32}{15}\,72^{1/4}=6.2142946779\ldots,
  \qquad \quartrad:=72^{1/4},
\end{equation}
from the dip; $\quartrad$ is the half-width of the optimal profile.
Throughout this subsection $\gamma\in\emb_{L}^{\bootstrap}$ satisfies \cref{eq:avoidorigin}, the parameter
$s^{*}$ is the envelope fold of \cref{lem:env}, and $q^{*}$, $D$, $\theta_{0}$, $\hat q^{*}$ are as in
\cref{prop:uncovered}.
We fix once and for all the angular threshold
\begin{equation}
  \label{eq:epschoice}
  \epsilon:=L^{-8/45},
\end{equation}
whose exponent is $\tfrac25\cdot\tfrac49$: it will transpire that the fold has depth $\Delta\asymp L^{-4/9}$,
and \cref{eq:epschoice} equalises at that depth the two error terms of \cref{prop:head} below.
Since $\cos\psi(s^{*})=0$ we have $\abs{h'(s^{*})}=1$; let $s_{-}:=\sup\{s<s^{*}:h'(s)=0\}$ and
$s_{+}:=\inf\{s>s^{*}:h'(s)=0\}$, which exist because $h'$ is continuous by \cref{rem:avoidorigin} and
vanishes at the extrema of $h$, and the depth is strictly monotone on $[s_{-},s_{+}]$: a fold is not a
turning point of the depth but a point of radial motion inside a single monotone descent.
Define
\begin{equation}
  \label{eq:Fdef}
  \begin{gathered}
    F_{\epsilon}:=\text{the maximal subinterval of }[s_{-},s_{+}]\text{ containing }s^{*}\\
    \text{on which }\abs{h'}\ge\sin\epsilon,
    \qquad
    \Delta:=\max_{F_{\epsilon}}h .
  \end{gathered}
\end{equation}
Since $\abs{h'(s^{*})}=1>\sin\epsilon$ while $h'(s_{\pm})=0$, the set $F_{\epsilon}$ is a nondegenerate
compact interval, contained in the open interval $(s_{-},s_{+})$, containing $s^{*}$ in its interior, and
carrying $\abs{h'}=\sin\epsilon$ at both endpoints; moreover
$\abs{F_{\epsilon}}\le\Delta/\sin\epsilon$, the depth being monotone with speed at least $\sin\epsilon$ along
it.
Note $\Delta\ge D$, and that neither is yet known to be small; smallness is \emph{established}, not assumed,
in \cref{lem:shallowfold} below, by running the blocking argument at an auxiliary shallow level.

\subsubsection*{The fold, both branches}

\begin{proposition}
  \label{prop:head}
  If $\Delta\le\tfrac12$, then
  \[
    \energy_{F_{\epsilon}}\ \ge\ \frac{\dgconst^{2}}{\Delta}
    \Bigl(1-C\epsilon^{3/2}-C\,\frac{\Delta}{\epsilon}\Bigr).
  \]
\end{proposition}

\begin{proof}
  On $F_{\epsilon}$ the depth is strictly monotone, so $h'=\sin\beta$ has a fixed sign and a continuous lift
  of $\beta$ along $F_{\epsilon}$ takes values in $(0,\pi)$ or in $(-\pi,0)$; after a reflection assume the
  former.
  At $s^{*}$, $\cos\beta=\cos\psi=0$, so $\beta(s^{*})=\tfrac\pi2$ exactly; at each endpoint of
  $F_{\epsilon}$, $\sin\beta=\sin\epsilon$, so $\beta\in\{\epsilon,\pi-\epsilon\}$ there.
  Apply \cref{lem:depthchart} separately to the two halves of $F_{\epsilon}$ on either side of $s^{*}$.
  Since $s^{*}$ is interior to $F_{\epsilon}$, each half is nondegenerate and has one endpoint at $s^{*}$ and
  the other at an endpoint of $F_{\epsilon}$; this is all that is asked of the two sides, and the depth
  $D=h(s^{*})$ of the fold plays no role, the endpoints of $F_{\epsilon}$ carrying $\abs{h'}=\sin\epsilon$
  whatever their depth.
  On each half $h$ is monotone; writing $x,y$ for the two depth ranges, both are positive, and monotonicity
  \emph{across} $s^{*}$ gives
  \[
    x+y=\int_{F_{\epsilon}}\abs{h'}\,ds=\max_{F_{\epsilon}}h-\min_{F_{\epsilon}}h\ \le\ \Delta ,
  \]
  which is the only property of the two ranges that is used.
  Each half has $\widehat G$-sweep, in absolute value, at least
  \[
    \min\Bigl(\widehat G(\tfrac\pi2)-\widehat G(\epsilon),\ \widehat G(\pi-\epsilon)-\widehat G(\tfrac\pi2)\Bigr)
    \;=\;\frac{\dgconst}2-\int_{0}^{\epsilon}\sqrt{\sin t}\,dt
    \;\ge\;\frac{\dgconst}2-\tfrac23\,\epsilon^{3/2},
  \]
  the two members being equal because $\sin(\pi-t)=\sin t$; and the correction term of
  \cref{eq:sweepidentity} is at most $\abs{F_{\epsilon}}/(1-\Delta)\le2\Delta/\sin\epsilon\le C\Delta/\epsilon$.
  Hence each half satisfies, by \cref{eq:sweepCS},
  $\energy_{\mathrm{half}}\ge a^{2}/x$ with $a:=\tfrac{\dgconst}2-\tfrac23\epsilon^{3/2}-C\Delta/\epsilon$; we
  may assume $a>0$, since otherwise the right-hand side of the claim is nonpositive once the constant $C$
  there is chosen large enough.
  No lower bound on $x$ or $y$ is needed, and none is available: a half of small depth range simply carries
  more energy, its sweep being unchanged.
  Finally
  $\energy_{F_{\epsilon}}\ge a^{2}\bigl(\tfrac1x+\tfrac1y\bigr)\ge\tfrac{4a^{2}}{x+y}\ge\tfrac{4a^{2}}{\Delta}$,
  and $4a^{2}\ge\dgconst^{2}\bigl(1-C\epsilon^{3/2}-C\Delta/\epsilon\bigr)$.
\end{proof}

The split into two halves is not a matter of convenience.
Applied to $F_{\epsilon}$ at once, \cref{eq:sweepCS} would control the \emph{signed} sweep
$\bigl[\widehat G(\beta)\bigr]_{\partial F_{\epsilon}}$, and that quantity vanishes whenever $\beta$ returns
to its initial value --- that is, whenever $\varphi'$ fails to change sign at the fold, which
\cref{def:lift} explicitly allows.
Splitting at $s^{*}$, where $\beta=\tfrac\pi2$, replaces one possibly vanishing sweep by two of size
$\dgconst/2$ each, and the harmonic mean $\tfrac1x+\tfrac1y\ge\tfrac4{x+y}$ restores the factor $4$; this is
where the constant $\dgconst^{2}$, and with it $\cstar$, is produced.

\subsubsection*{Blocking}

The barrier is assembled from the fold and the ray above it; the level at which it blocks is a parameter,
because the smallness of $\Delta$ is not yet available.
Throughout, $s^{*}$ is the envelope fold of \cref{lem:env} and $q^{*}=\gamma(s^{*})$, $D=h(s^{*})$,
$\theta_{0}=\varphi(s^{*})$, $\hat q^{*}=q^{*}/\abs{q^{*}}$ are as in \cref{prop:uncovered}.

\begin{lemma}
  \label{lem:blocking}
  Let $0<\delta\le\min\bigl(\Delta,\tfrac1{10}\sin\epsilon\bigr)$ and set
  \[
    \mathcal C_{\delta}:=
    \begin{cases}
      \bigl(F_{\epsilon}\cap\{h\le\delta\}\bigr)\ \cup\ [q^{*},\hat q^{*}], & D\le\delta,\\[3pt]
      [q^{*},\hat q^{*}]\cap\{h\le\delta\}, & D>\delta,
    \end{cases}
  \]
  where, by abuse of notation, $F_{\epsilon}\cap\{h\le\delta\}$ denotes the image under $\gamma$ of the
  corresponding parameter set.
  Then $\mathcal C_{\delta}$ is a continuum contained in $\{h\le\delta\}$, joining $\partial\disk$ to
  $\{h=\delta\}$, meeting $\gamma([0,L])$ only in $\gamma(F_{\epsilon})$, and of angular width at most
  \[
    \zeta_{\delta}\ :=\ \frac{2\delta}{\sin\epsilon}\ \le\ \tfrac15\,;
  \]
  and every connected arc of $\gamma\setminus F_{\epsilon}$ contained in $\{h\le\delta\}$ has angular
  oscillation at most $2\pi+\zeta_{\delta}$.
\end{lemma}

\begin{proof}
  \emph{The barrier.}
  Suppose first $D\le\delta$.
  The depth is monotone along $F_{\epsilon}$, so the parameter set $\{h\le\delta\}\cap F_{\epsilon}$ is a
  compact subinterval containing the shallow endpoint of $F_{\epsilon}$ and, since $h(s^{*})=D\le\delta$, the
  fold $s^{*}$; as $\delta\le\Delta=\max_{F_{\epsilon}}h$, its deep endpoint has depth exactly $\delta$, by
  the intermediate value theorem.
  Its image is a continuum through $q^{*}$; the segment $[q^{*},\hat q^{*}]$ is a continuum through $q^{*}$
  with depths in $[0,D]\subset[0,\delta]$; their union $\mathcal C_{\delta}$ is a continuum in
  $\{h\le\delta\}$ meeting $\partial\disk$ at $\hat q^{*}$ and $\{h=\delta\}$ at the deep endpoint.
  If instead $D>\delta$, then $\mathcal C_{\delta}$ is the outer piece of the segment, a radial segment with
  depths ranging over $[0,\delta]$, again a continuum joining $\partial\disk$ to $\{h=\delta\}$.

  \emph{The barrier meets the trace only in the fold.}
  If $D\le\delta$: by \cref{prop:uncovered} the segment $(q^{*},\hat q^{*}]$ misses the trace.
  If $D>\delta$: every point of $\mathcal C_{\delta}$ has radius at least $1-\delta>1-D=\abs{q^{*}}$, so lies
  in $(q^{*},\hat q^{*}]$ and misses the trace by \cref{prop:uncovered}.

  \emph{The width.}
  The segment projects to the single angle $\theta_{0}$.
  The parameter set $F_{\epsilon}\cap\{h\le\delta\}$ has length at most $\delta/\sin\epsilon$, the depth
  being monotone along it with speed at least $\sin\epsilon$ and range at most $\delta$; and on it
  $\abs{\varphi'}\le(1-\delta)^{-1}\le\tfrac{10}9$ by \cref{eq:framerelations}.
  Hence the width is at most $\tfrac{10}9\,\delta/\sin\epsilon\le\zeta_{\delta}$, and
  $\delta\le\tfrac1{10}\sin\epsilon$ gives $\zeta_{\delta}\le\tfrac15$.

  \emph{Blocking.}
  The set $\bar A_{\delta}:=\{h\le\delta\}=\{1-\delta\le\abs z\le1\}$ is a closed annulus about the origin;
  identify its universal cover with the strip $\mathbb{R}\times[1-\delta,1]$ through the polar coordinates
  $(\varphi,r)$.
  Since $\zeta_{\delta}<\pi$, the barrier lies in a closed sector of angular width less than $2\pi$, which is
  evenly covered by the projection \cite[\S1.3]{Hatcher2002}; so the preimage of $\mathcal C_{\delta}$ is the
  disjoint union of translates $\mathcal C^{0}+2\pi n$, $n\in\mathbb{Z}$, each a continuum joining the two
  edges of the strip, with angular projection an interval
  $[\alpha+2\pi n,\alpha'+2\pi n]$ of length at most $\zeta_{\delta}$.
  Let now $\Gamma$ be a connected arc of $\gamma\setminus F_{\epsilon}$ contained in $\bar A_{\delta}$.
  Since $\gamma$ is injective, $\Gamma$ is disjoint from $\gamma(F_{\epsilon})$, hence from
  $\mathcal C_{\delta}$, and its lift, given by the global lift of $\varphi$, is a path in the strip avoiding
  every translate.
  Suppose its angular oscillation exceeded $2\pi+\zeta_{\delta}$: it would contain parameters $p,q$ with
  $\varphi(q)-\varphi(p)>2\pi+\zeta_{\delta}$, and since the interval
  $\bigl(\varphi(p),\varphi(q)-\zeta_{\delta}\bigr)$ has length exceeding $2\pi$ it contains
  $A:=\alpha+2\pi n$ for some $n$; let $B:=\alpha'+2\pi n$ be the right endpoint of that translate's angular
  projection, so that $B\le A+\zeta_{\delta}<\varphi(q)$.
  If $B=A$, which is the case whenever the barrier is a radial segment, the translate is a continuum in the
  vertical segment $\{A\}\times[1-\delta,1]$ joining its two ends, hence the whole segment; since
  $\varphi(p)<A<\varphi(q)$, the lifted arc crosses the line $\{\varphi=A\}$, and it does so inside the strip:
  a contradiction.
  If $A<B$, travel along the lifted arc from $p$ towards $q$ and let $t_{2}$ be the first parameter with
  $\varphi(t_{2})=B$ and $t_{1}$ the last parameter before $t_{2}$ with $\varphi(t_{1})=A$; on $[t_{1},t_{2}]$
  the lifted arc is a path in the rectangle $[A,B]\times[1-\delta,1]$ joining its two vertical edges, while
  the translate $\mathcal C^{0}+2\pi n$ is a continuum in the same rectangle joining its two horizontal
  edges, and the two are disjoint.
  But a continuum joining the two horizontal edges of a rectangle meets every path joining its two vertical
  edges (see \cite[Ch.~V]{Newman1951}, or \cite{Gale1979} via the Hex theorem): a contradiction.
\end{proof}

\subsubsection*{The passage count}

\begin{lemma}
  \label{lem:separation}
  Let $\delta$ be as in \cref{lem:blocking}, let $T:=2\quartrad\,\delta^{1/4}$, and let
  $\sigma_{1},\dots,\sigma_{M}$ be a maximal $T$-separated subset of $\{h\ge\delta\}$; such a subset exists
  and is finite, a $T$-separated subset of $\mathbb{R}/L\mathbb{Z}$ having at most $L/T$ elements, and it is
  nonempty because $\Delta\ge\delta$ is attained on $F_{\epsilon}$.
  Then, for $L\ge L_{0}(\bootstrap)$,
  \begin{equation}
    \label{eq:countmaster}
    M\ \ge\ \frac{L-C\max(\excess,1)}{2\pi+\zeta_{\delta}+2T}\ \ge\ \frac{L}{14}\,;
  \end{equation}
  and if moreover $\zeta_{\delta}\le T$ and $L^{-1/2}\le T\le1$, then
  $M\ge\dfrac{L}{2\pi}\bigl(1-CT\bigr)$.
\end{lemma}

\begin{proof}
  By maximality, every point of $\{h\ge\delta\}$ lies within distance $T$ of some $\sigma_{k}$, so the
  interiors $P_{j}:=(\sigma_{j}+T,\sigma_{j+1}-T)$, $j$ cyclic and $P_{j}$ possibly empty, are contained in
  $\{h<\delta\}$, and $\sum_{j}\abs{P_{j}}\ge L-2TM$.

  \emph{Lower bound for the angular variation over the interiors.}
  On $\{h<\delta\}$, \cref{eq:framerelations} gives
  $\abs{\varphi'}=\sqrt{1-(h')^{2}}/(1-h)\ge1-(h')^{2}$, so by \cref{cor:hprime}
  \begin{equation}
    \label{eq:countlower}
    \sum_{j}\int_{P_{j}}\abs{\varphi'}\,ds\ \ge\ \sum_{j}\abs{P_{j}}-\int_{0}^{L}(h')^{2}\,ds
    \ \ge\ L-2TM-6\excess .
  \end{equation}

  \emph{Upper bound.}
  Let $W:=F_{\epsilon}\cap\{h<\delta\}$, a single parameter interval of length at most $\delta/\sin\epsilon$,
  and let $\Pi_{1},\dots,\Pi_{N}$, $N\le C\max(\excess,1)$, be the partition of \cref{cor:arcs}\,(i),
  so that $\sgn\sweepdens=\sgn\varphi'$ is constant off
  $Z:=\{\abs{\sweepdens}<\tfrac12\}$ on each $\Pi_{k}$, with $\abs Z\le C\excess$.
  Cut the disjoint intervals $P_{j}\setminus W$ --- at most $M+2$ intervals in all, the removal of the single
  interval $W$ creating at most two new components --- at the $N$ endpoints of the blocks; this produces at
  most $M+2+C\max(\excess,1)$ pieces $Q$, each a parameter interval disjoint from $F_{\epsilon}$, contained
  in $\{h<\delta\}$ and in a single block.
  For each such piece, \cref{lem:blocking} bounds the oscillation, the sign of $\varphi'$ is constant off
  $Z$, and $\abs{\varphi'}\le2$ on $Z\cap\{h\le\tfrac12\}$ by \cref{cor:arcs}\,(ii); hence
  \[
    \int_{Q}\abs{\varphi'}\,ds
    \ \le\ \Bigl|\int_{Q}\varphi'\,ds\Bigr|+2\int_{Q\cap Z}\abs{\varphi'}\,ds
    \ \le\ \operatorname*{osc}_{Q}\varphi+4\abs{Q\cap Z}
    \ \le\ 2\pi+\zeta_{\delta}+4\abs{Q\cap Z} .
  \]
  On $W$ itself, $\abs{\varphi'}\le\tfrac{10}9$ and $\abs W\le\delta/\sin\epsilon$, so
  $\int_{\bigcup_{j}P_{j}\cap W}\abs{\varphi'}\le2\delta/\sin\epsilon=\zeta_{\delta}$.
  Summing over the pieces,
  \begin{equation}
    \label{eq:countupper}
    \begin{aligned}
      \sum_{j}\int_{P_{j}}\abs{\varphi'}\,ds
      &\ \le\ \bigl(M+2+C\max(\excess,1)\bigr)\bigl(2\pi+\zeta_{\delta}\bigr)+4\abs{Z}+\zeta_{\delta}\\
      &\ \le\ M\bigl(2\pi+\zeta_{\delta}\bigr)+C\max(\excess,1) .
    \end{aligned}
  \end{equation}

  Comparing \cref{eq:countlower,eq:countupper} gives
  $M(2\pi+\zeta_{\delta}+2T)\ge L-C\max(\excess,1)$, which is the first inequality of
  \cref{eq:countmaster}; the second follows from $\zeta_{\delta}\le\tfrac15$,
  $T\le2\quartrad\cdot10^{-1/4}\le\tfrac{33}{10}$ and \cref{eq:bootstrap}, for $L\ge L_{0}(\bootstrap)$.
  Under the additional hypotheses,
  \[
    M\ \ge\ \frac{L-C\max(\excess,1)}{2\pi+3T}
    \ \ge\ \frac{L}{2\pi}\Bigl(1-\frac{3T}{2\pi}\Bigr)\Bigl(1-C\,\frac{\max(\excess,1)}{L}\Bigr)
    \ \ge\ \frac{L}{2\pi}\bigl(1-CT\bigr),
  \]
  since $\max(\excess,1)/L\le\bootstrap L^{-5/9}\le\bootstrap\,T L^{-1/18}\le CT$ by \cref{eq:bootstrap} and
  $T\ge L^{-1/2}$.
\end{proof}

\subsubsection*{A priori bounds on the fold depth}

\begin{lemma}
  \label{lem:shallowfold}
  For $L\ge L_{0}(\bootstrap)$,
  \[
    \Delta\ \le\ C(\bootstrap)\,L^{-5/18} .
  \]
  In particular $\Delta<\tfrac1{10}\sin\epsilon$, so the level $\delta=\Delta$ is admissible in
  \cref{lem:blocking,lem:separation}, and $\Delta\le\tfrac12$, so \cref{prop:head} applies.
\end{lemma}

\begin{proof}
  Let $\delta:=\min\bigl(\Delta,\tfrac1{10}\sin\epsilon\bigr)$, an admissible level, and let
  $\sigma_{1},\dots,\sigma_{M}$ be as in \cref{lem:separation}, so that $M\ge L/14$.
  The separation $T=2\quartrad\delta^{1/4}$ exceeds $\delta$, since $\delta\le1$ and $2\quartrad>1$; the depth
  is $1$-Lipschitz with $h(\sigma_{j})\ge\delta$, so $h\ge\delta/2$ on the pairwise disjoint intervals
  $[\sigma_{j}-\tfrac\delta2,\sigma_{j}+\tfrac\delta2]$, and \cref{eq:B1,eq:depthbound} give
  \[
    \excess\ \ge\ \int_{0}^{L}h\,ds\ \ge\ M\,\frac{\delta^{2}}{2}\ \ge\ \frac{L\,\delta^{2}}{28},
    \qquad\text{whence}\qquad
    \delta\ \le\ \bigl(28\,\bootstrap\bigr)^{1/2}L^{-5/18}
  \]
  by \cref{eq:bootstrap}.
  Since $\tfrac1{10}\sin\epsilon\ge\tfrac1{20}\epsilon=\tfrac1{20}L^{-8/45}$ and $\tfrac8{45}<\tfrac5{18}$,
  for $L\ge L_{0}(\bootstrap)$ the right-hand side is smaller than $\tfrac1{10}\sin\epsilon$; the minimum is
  therefore not attained at $\tfrac1{10}\sin\epsilon$, so $\delta=\Delta$, and the bound follows.
\end{proof}

The choice $\delta=\Delta$ is made from now on: $T=2\quartrad\Delta^{1/4}$,
$\zeta:=\zeta_{\Delta}=2\Delta/\sin\epsilon$, and $\sigma_{1},\dots,\sigma_{M}$ denote a maximal
$T$-separated subset of $\{h\ge\Delta\}$.
The hypotheses of the sharp count in \cref{lem:separation} will be verified in the proof of
\cref{thm:lowersharp}, where the bootstrap also supplies $cL^{-4/9}\le\Delta\le CL^{-4/9}$.

\subsubsection*{The dip, with the sharp constant}

\begin{lemma}
  \label{lem:dipsharp}
  Let $I$ be an interval, $u\in W^{2,2}(I)$ with $u\ge0$, and $s_{0}\in I$ with $u(s_{0})\ge A$ and
  $\dist(s_{0},\partial I)\ge\quartrad A^{1/4}$.
  Then
  \[
    \int_{I}\Bigl(2u+(u'')^{2}\Bigr)dx\ \ge\ \dipconst\,A^{5/4}.
  \]
  No condition is imposed on $u$ at $\partial I$.
\end{lemma}

\begin{proof}
  Let $w$ be the quartic profile centred at $s_{0}$ with peak $A$: with $t=\quartrad A^{1/4}-\abs{x-s_{0}}$,
  $w=\alpha t^{3}-t^{4}/24$ and $\alpha=\quartrad A^{1/4}/18$, so that
  $\operatorname{supp}w=[s_{0}-\quartrad A^{1/4},s_{0}+\quartrad A^{1/4}]=:[s_{-},s_{+}]\subset I$,
  $w''''=-1$ on the interior of its support, and $F(w)=\dipconst A^{5/4}$, where
  $F(u):=\int_{I}(2u+(u'')^{2})$.
  By convexity,
  \[
    F(u)-F(w)=\int 2(u-w)+\int(u''-w'')^{2}+2\int w''(u-w)''
  \]
  \[
    \ \ge\ \int 2(u-w)+2\int w''(u-w)'' .
  \]
  Two integrations by parts give $\int w''(u-w)''=\int w''''(u-w)$ in the sense of distributions: the first
  boundary term carries $w''$, which vanishes at $s_{\pm}$ and beyond, and the second carries $w'''$, which
  vanishes off the support; when $s_{\pm}$ coincides with an endpoint of $I$ the corresponding boundary
  contribution equals the edge atom below and the identity is unchanged.
  Here
  \[
    w''''=-\mathbb 1_{\operatorname{supp}w}
      +\tfrac43\quartrad A^{1/4}\,\delta_{s_{0}}
      +\tfrac13\quartrad A^{1/4}\bigl(\delta_{s_{-}}+\delta_{s_{+}}\bigr),
  \]
  the atoms arising because $w'''$ jumps at the peak and at the two edges of the support, by
  $\tfrac43\quartrad A^{1/4}$ and $\tfrac13\quartrad A^{1/4}$ respectively; all three are positive.
  Hence
  \[
    F(u)-F(w)\ \ge\ 2\int_{I\setminus\operatorname{supp}w}u
    +\tfrac83\quartrad A^{1/4}\bigl(u(s_{0})-A\bigr)
    +\tfrac23\quartrad A^{1/4}\bigl(u(s_{-})+u(s_{+})\bigr),
  \]
  every term being nonnegative, the last because $w(s_{\pm})=0$ and $u\ge0$. \qedhere
\end{proof}

\begin{remark}
  \label{rem:freeends}
  The hypothesis on $\dist(s_{0},\partial I)$ cannot be dropped: half of the profile is admissible if the
  peak is allowed to sit at $\partial I$, and the constant halves.
  In the application it costs nothing, the peaks being $\asymp2\pi$ apart while
  $\quartrad\Delta^{1/4}\asymp L^{-1/9}$.
  What matters is that \emph{no} condition is imposed at $\partial I$: a passage of the wire need not
  return to any prescribed baseline, and requiring it to do so is what would force the discarding of too
  many turns.
\end{remark}

\begin{corollary}
  \label{cor:dipsum}
  Let $\delta=\Delta$, admissible by \cref{lem:shallowfold}, let $\sigma_{1},\dots,\sigma_{M}$ be as in
  \cref{lem:separation}, and let $I_{j}$ be the intervals cut at the midpoints of consecutive $\sigma$'s.
  Then, for every $j$,
  \[
    \int_{I_{j}}\Bigl(2h+(h'')^{2}\Bigr)ds\ \ge\ \dipconst\,\Delta^{5/4} ;
  \]
  the $I_{j}$ being pairwise disjoint, the inequality may be summed over any subfamily.
\end{corollary}

\begin{proof}
  Each window satisfies $\dist(\sigma_{j},\partial I_{j})\ge T/2=\quartrad\Delta^{1/4}$, the midpoints being
  at distance at least $T/2$ from the two peaks they separate; $h\in W^{2,2}(I_{j})$ with $h\ge0$ by
  \cref{rem:avoidorigin}, and $h(\sigma_{j})\ge\Delta$; apply \cref{lem:dipsharp} with $u=h$, $A=\Delta$,
  $s_{0}=\sigma_{j}$.
\end{proof}

\subsubsection*{The sharp lower bound}

\begin{theorem}
  \label{thm:lowersharp}
  There are $C$ and $L_{0}$ such that, for every $L\ge L_{0}$,
  \[
    \minen_{L}\ \ge\ L+\cstar\,L^{4/9}-C\,L^{1/3},
    \qquad
    \cstar:=\frac95\bigl(\dgconst^{2}\bigr)^{5/9}\Bigl(\frac{5\dipconst}{8\pi}\Bigr)^{4/9}
    =5.223049104\ldots
  \]
\end{theorem}

\begin{proof}
  By \cref{rem:reduction} it suffices to bound $\excess$ from below on $\emb_{L}^{\bootstrap}$, and by
  \cref{rem:avoidorigin} we may assume \cref{eq:avoidorigin}; all constants below depend only on
  $\bootstrap$, which is itself absolute.
  Let $\Delta$, $T=2\quartrad\Delta^{1/4}$, $\zeta$ and the windows $I_{1},\dots,I_{M}$ at level
  $\delta=\Delta$ be as fixed after \cref{lem:shallowfold}, which makes that level admissible, and set
  \begin{equation}
    \label{eq:levelchoice}
    \varsigma:=6\,\Delta^{5/8},\qquad \lambda:=\varsigma^{1/2} .
  \end{equation}
  By \cref{lem:shallowfold}, $\Delta\le C L^{-5/18}$, so for $L\ge L_{0}$ we have $\varsigma\le\tfrac1{10}$,
  $\lambda\le\tfrac1{22}$, and $\varsigma\le T$, the last because $6\Delta^{5/8}\le2\quartrad\Delta^{1/4}$
  whenever $\Delta^{3/8}\le\quartrad/3$; and $\varsigma^{2}/4=9\Delta^{5/4}\ge\dipconst\Delta^{5/4}$, since
  $\dipconst<9$ by \cref{eq:dipconst}.
  The level $\varsigma$ is tied to $\Delta$ rather than to $L$ because the two-sided bound
  $\Delta\asymp L^{-4/9}$ is only available after Step~3.

  \emph{Step 1: $\Delta\ge c\,L^{-4/9}$.}
  By \cref{eq:arcexcess,eq:boundaryterm}, and since $\abs{h'}=\sin\epsilon$ at the endpoints of
  $F_{\epsilon}$ while $\abs{F_{\epsilon}}\le\Delta/\sin\epsilon$,
  \begin{equation}
    \label{eq:headcharge}
    \widehat{\excess}_{F_{\epsilon}}
    \ \ge\ \energy_{F_{\epsilon}}-\abs{F_{\epsilon}}-2\bigl(\abs{h'}(a)+\abs{h'}(b)\bigr)
    \ \ge\ \frac{\dgconst^{2}}{\Delta}\Bigl(1-C\epsilon^{3/2}-C\frac{\Delta}{\epsilon}\Bigr)
      -\frac{2\Delta}{\epsilon}-4\epsilon ,
  \end{equation}
  by \cref{prop:head}, whose hypothesis $\Delta\le\tfrac12$ holds by \cref{lem:shallowfold}, and by
  $\sin\epsilon\ge\tfrac\epsilon2$.
  By \cref{lem:shallowfold} and \cref{eq:epschoice}, $\Delta/\epsilon\le CL^{-5/18+8/45}=CL^{-1/10}$ and
  $\epsilon^{3/2}=L^{-4/15}$, so for $L\ge L_{0}$ the parenthesis is at least $\tfrac12$; and
  $\widehat{\excess}_{F_{\epsilon}}\le\excess$ by \cref{eq:additivity} applied to the single arc
  $F_{\epsilon}$.
  Hence, by \cref{eq:bootstrap},
  \[
    \frac{\dgconst^{2}}{2\Delta}\ \le\ \excess+5\ \le\ 2\bootstrap L^{4/9},
    \qquad\text{that is}\qquad
    \Delta\ \ge\ \frac{\dgconst^{2}}{4\bootstrap}\,L^{-4/9}=:c_{4}L^{-4/9} .
  \]

  \emph{Step 2: the floor of a window.}
  The windows meeting $F_{\epsilon}$ number at most $\abs{F_{\epsilon}}/T+2
  \le\Delta^{3/4}/(2\quartrad\sin\epsilon)+2\le C$, by \cref{lem:shallowfold} and
  $\tfrac34\cdot\tfrac5{18}>\tfrac8{45}$; we set them aside.
  We claim that every other window satisfies
  \begin{equation}
    \label{eq:windowfloor}
    \widehat{\excess}_{I_{j}}\ \ge\ \bigl(1-11\,\varsigma^{1/2}\bigr)\,\dipconst\,\Delta^{5/4} .
  \end{equation}
  Two cases, according to the size of $H_{j}:=\max_{I_{j}}h$.
  If $H_{j}>\varsigma$, let $p$ be a point of $I_{j}$ with $h(p)>\varsigma$; since $\abs{I_{j}}\ge T\ge\varsigma$,
  the two halves of $I_{j}$ on either side of $\sigma_{j}$ having length at least $T/2$ each
  (\cref{cor:dipsum}), the longer side of $p$ inside $I_{j}$ has length at least $\varsigma/2$, and $h\ge\varsigma/2$ on its initial
  segment of length $\varsigma/2$ by the Lipschitz bound; hence, by \cref{cor:localbudgets},
  \[
    \widehat{\excess}_{I_{j}}\ \ge\ \int_{I_{j}}h\,ds\ \ge\ \frac{\varsigma^{2}}4\ =\ 9\Delta^{5/4}
    \ \ge\ \dipconst\Delta^{5/4},
  \]
  which is more than \cref{eq:windowfloor} asks: a deep window is not cheaper than a shallow one, and need
  not be discarded.
  If $H_{j}\le\varsigma$, then $I_{j}\subset\{h\le\tfrac1{10}\}$, and \cref{prop:conversion} together with
  \cref{cor:dipsum} gives
  $\widehat{\excess}_{I_{j}}\ge(1-\lambda)\dipconst\Delta^{5/4}-10\lambda^{-1}\varsigma\,\widehat{\excess}_{I_{j}}$,
  that is
  \[
    \widehat{\excess}_{I_{j}}\ \ge\ \frac{1-\lambda}{1+10\lambda^{-1}\varsigma}\,\dipconst\Delta^{5/4}
    \ \ge\ \bigl(1-\lambda-10\lambda^{-1}\varsigma\bigr)\dipconst\Delta^{5/4}
    \ =\ \bigl(1-11\varsigma^{1/2}\bigr)\dipconst\Delta^{5/4},
  \]
  by the choice $\lambda=\varsigma^{1/2}$.
  This proves \cref{eq:windowfloor}; note that it costs no window at all, and that the relative loss
  $11\varsigma^{1/2}\asymp\Delta^{5/16}$ is the whole price of converting the charge into the model functional.

  \emph{Step 3: $\Delta\le C\,L^{-4/9}$.}
  The windows are pairwise disjoint, so summing \cref{eq:windowfloor} over the at least $M-C$ windows disjoint
  from $F_{\epsilon}$ and using \cref{eq:additivity,eq:countmaster,eq:bootstrap}, with $\lambda\le\tfrac1{22}$,
  \[
    \bootstrap L^{4/9}\ \ge\ \excess\ \ge\ \Bigl(\frac{L}{14}-C\Bigr)\cdot\tfrac12\,\dipconst\Delta^{5/4},
    \qquad\text{whence}\qquad
    \Delta\ \le\ C_{5}\,L^{-4/9} .
  \]

  \emph{Step 4: the sharp assembly.}
  By Steps 1 and 3, $c_{4}L^{-4/9}\le\Delta\le C_{5}L^{-4/9}$, so
  $T=2\quartrad\Delta^{1/4}$ satisfies $L^{-1/2}\le cL^{-1/9}\le T\le CL^{-1/9}\le1$ and
  $\zeta=2\Delta/\sin\epsilon\le CL^{-4/9+8/45}=CL^{-4/15}\le T$ for $L\ge L_{0}$: the sharp branch of
  \cref{lem:separation} applies and gives $M\ge\tfrac{L}{2\pi}(1-CT)$, so that at least
  $\tfrac{L}{2\pi}(1-CL^{-1/9})-C$ windows are disjoint from $F_{\epsilon}$.
  These windows and the arc $F_{\epsilon}$ are pairwise disjoint, so \cref{eq:additivity},
  \cref{eq:headcharge} and \cref{eq:windowfloor} give
  \begin{align*}
    \excess\ &\ge\ \widehat{\excess}_{F_{\epsilon}}+\sum_{I_{j}\cap F_{\epsilon}=\emptyset}\widehat{\excess}_{I_{j}}\\
    &\ge\ \frac{\dgconst^{2}}{\Delta}\Bigl(1-C\epsilon^{3/2}-C\frac{\Delta}{\epsilon}\Bigr)-C\\
    &\qquad\ +\ \Bigl(\frac{L}{2\pi}\bigl(1-CL^{-1/9}\bigr)-C\Bigr)\bigl(1-11\varsigma^{1/2}\bigr)\,\dipconst\,\Delta^{5/4}\\
    &\ge\ \frac{\dgconst^{2}}{\Delta}+\frac{\dipconst}{2\pi}\,L\,\Delta^{5/4}-CL^{1/3},
  \end{align*}
  where the errors were absorbed as follows.
  By Step 1, both $\dgconst^{2}\epsilon^{3/2}/\Delta$ and $\dgconst^{2}/\epsilon$ are at most
  $CL^{8/45}$, and $\tfrac8{45}<\tfrac13$; by Step 3, $L\Delta^{5/4}\le CL^{4/9}$, so the separation loss
  contributes at most $CL^{4/9}\cdot L^{-1/9}=CL^{1/3}$ and the conversion loss at most
  $CL^{4/9}\varsigma^{1/2}\le CL^{4/9}\Delta^{5/16}\le CL^{4/9-5/36}=CL^{11/36}\le CL^{1/3}$; and
  $C\dipconst\Delta^{5/4}\le C$.
  Finally, for every $\Delta>0$,
  \[
    \frac{a}{\Delta}+b\,\Delta^{5/4}\ \ge\ \frac95\,a^{5/9}\Bigl(\frac{5b}{4}\Bigr)^{4/9},
  \]
  the minimum being attained at $\Delta=(4a/5b)^{4/9}$; with $a=\dgconst^{2}$ and
  $b=\dipconst L/2\pi$ the right-hand side is $\cstar L^{4/9}$, and the theorem follows from
  \cref{rem:reduction}.
\end{proof}

The proof gives more than the inequality it was written for: it shows that the mechanism priced above is
forced on every competitor of small excess, and not merely available to good ones.

\begin{corollary}
  \label{cor:deltarange}
  There are $c>0$, $C$ and $L_{0}$ such that every $\gamma\in\emb_{L}^{\bootstrap}$ satisfying
  \cref{eq:avoidorigin} with $L\ge L_{0}$ has
  \[
    c\,L^{-4/9}\ \le\ \Delta\ \le\ C\,L^{-4/9},
    \qquad
    \widehat{\excess}_{F_{\epsilon}}\ \ge\ c\,L^{4/9} .
  \]
\end{corollary}

\begin{proof}
  The two-sided bound on $\Delta$ is the content of Steps~1 and~3 of the preceding proof, which use nothing
  about $\gamma$ beyond \cref{eq:bootstrap,eq:avoidorigin}.
  By \cref{eq:headcharge} and the absorption of its error terms in Step~1,
  $\widehat{\excess}_{F_{\epsilon}}\ge\dgconst^{2}/(2\Delta)-5$, and the upper bound on $\Delta$ concludes.
\end{proof}

So the envelope fold of any competitor of small excess sits at the depth $\asymp L^{-4/9}$ realised by the
construction of \cref{sec:upper}; and, $\excess$ being at most $\bootstrap L^{4/9}$ there, that single fold
already carries a definite fraction of the whole excess.
This is recorded rather than used: nothing below appeals to it.

\begin{remark}
  \label{rem:whyLthird}
  The error $O(L^{1/3})$ has a single source.
  In the proof above the fold is priced to within $CL^{8/45}$, the conversion of the charge of a window into the
  model functional loses the relative amount $11\varsigma^{1/2}\asymp\Delta^{5/16}$, that is $CL^{11/36}$ in
  absolute terms, and only the $O(1)$ windows meeting the fold are set aside; all of this lies strictly below
  $L^{1/3}$.
  What lands on $L^{1/3}$ is the separation loss $CT\asymp\Delta^{1/4}$ of \cref{lem:separation} alone: the
  calibration of \cref{lem:dipsharp} requires each peak to sit at distance $\quartrad\Delta^{1/4}$ from the ends
  of its window, so the count of \cref{lem:separation} excludes a collar of width $T\asymp\Delta^{1/4}$ around
  each peak, and an arc can legitimately gain an angle $\lesssim T$ while passing beneath the barrier inside a
  collar.
  The count is therefore known only to the relative precision $O(\Delta^{1/4})=O(L^{-1/9})$, whereas the
  second term of \cref{eq:expansion} requires $M=\tfrac{L}{2\pi}\bigl(1+O(\Delta^{1/2})\bigr)$.
  The loss is an artefact of the method: in the competitor of \cref{sec:upper} the peaks are $\asymp2\pi$
  apart and no angle is wasted.
  Nor can it be avoided by weakening \cref{lem:dipsharp}, whose constant genuinely halves in that case
  (\cref{rem:freeends}).

  The conversion loss, by contrast, is not a genuine limitation, and it is worth recording what the master
  identity gives when nothing is discarded.
  By \cref{eq:master2}, $1-r^{2}=2h-h^{2}$, $\langle\gamma,\tau\rangle^{2}=(1-h)^{2}(h')^{2}$ and
  \cref{eq:exactidentity}, on an arc $J\subset\{h\le\tfrac1{10}\}$,
  \[
    \widehat{\excess}_{J}
    \;=\;\int_{J}\bigl(2h-h^{2}\bigr)ds+\int_{J}(1-h)^{2}(h')^{2}\,ds
    +\int_{J}\frac{\bigl(h''+e\bigr)^{2}}{1-(h')^{2}}\,ds ,
  \]
  and since $e=2h+O\bigl(h^{2}+h(h')^{2}\bigr)$ and $\int_{J}2h''h=2[hh']_{\partial J}-2\int_{J}(h')^{2}$,
  formally
  \[
    \widehat{\excess}_{J}
    \;=\;\int_{J}\Bigl(2h+(h'')^{2}-3(h')^{2}\Bigr)ds\;+\;4\bigl[hh'\bigr]_{\partial J}\;+\;E_{J},
  \]
  where $E_{J}$ collects terms quadratic in $h$, of the form $h(h')^{2}$, and of the form $(h')^{2}(h'')^{2}$,
  each of order $\Delta^{9/4}$ on a window of the competitor of \cref{sec:upper}.
  The per-window functional is thus exactly that of \cref{lem:dipcost}, including the coefficient $-3$: the
  charge of an arc \emph{is} the second-order expansion of its excess, and the identity \cref{eq:master2}
  supplies it without any assumption of graphicality.
  What is missing for the second-order term from below is therefore not the conversion but two other steps: a
  lower bound on the angular spacing of deep points in consecutive turns, or a pricing of the angle gained
  inside a collar; and a perturbative form of the calibration carrying the $-3\int(u')^{2}$ term, of value
  $\dipconst A^{5/4}-3\dipcorr A^{7/4}+O(A^{9/4})$.
  Neither is carried out here.
\end{remark}

\begin{conjecture}
  \label{conj:ctwo}
  With $\ctwo=-3\dipcorr\depthconst^{7/4}/2\pi$ as in \cref{thm:upper},
  \[
    \minen_{L}\ =\ L+\cstar\,L^{4/9}+\ctwo\,L^{2/9}+O(1)
    \qquad\text{as }L\to\infty .
  \]
  The upper bound is \cref{thm:upper}; the missing half is the lower bound, whose obstruction is localised in
  \cref{rem:whyLthird}.
\end{conjecture}

\begin{proof}[Proof of \cref{thm:sharpmain}]
  Combine \cref{thm:lowersharp,thm:upper}.
\end{proof}

%% file: sec6-remarks.tex
\section{Remarks and open problems}
\label{sec:remarks}

\subsection*{What is missing from the expansion}

Of the four terms of \cref{eq:expansion} we have established the first two, and the third from above only.
Two obstructions remain, of quite different character.

The first is the error $O(L^{1/3})$ in \cref{thm:sharpmain}, which has a single source: the separation loss
of \cref{lem:separation}, which knows the passage count only to relative precision $O(\Delta^{1/4})$, against
the $O(\Delta^{1/2})$ that the second term of \cref{eq:expansion} requires.
Every other loss of the proof of \cref{thm:lowersharp} lies strictly below $L^{1/3}$, and \cref{rem:whyLthird}
records that the master identity \cref{eq:master2} already delivers, per window, the second-order functional
$\int(2h+(h'')^{2}-3(h')^{2})$ of \cref{lem:dipcost}; what is missing is the calibration of that functional
and the sharper count.
This is \cref{conj:ctwo}; the upper bound already gives $\ctwo=-3\dipcorr\depthconst^{7/4}/2\pi$.

The constant term $c_{0}$ is harder, and not only longer.
It requires the fold to be priced to relative accuracy $O(\Delta)$, and at that precision the disk is not a
strip: the correction term in \cref{eq:sweepidentity}, which we bounded by $\abs{J}/(1-h)$, is itself of
order one and must be evaluated.
That amounts to perturbing the free elastica of \cref{lem:fold} to first order in the ambient curvature, a
computation we have not attempted.
The three contributions to $c_{0}$ --- from the fold, from the third-order term of \cref{lem:dipcost}, and
from the interior connector discussed next --- would then also have to be shown not to interact at that
order.

\subsection*{The interior connector}

A closed competitor cannot be contained in a neighbourhood of $\partial\disk$ on which the sense of rotation
about the origin is constant: by \cref{lem:omega} the sweep density $\sweepdens$ is continuous, and the
competitor of \cref{sec:upper} shows that both signs must in fact occur on sets of length $\asymp L$.
Where the sign changes, the curve crosses the interior of the disk, and the crossing is not free.
Let $\conncl$ be the class of arcs contained in $\overline\disk$ which meet $\partial\disk$ tangentially at
their two endpoints and reverse the sense of rotation about the origin, and let $\econn$ be the infimum of
the excess over $\conncl$; a short computation, which we omit, gives $2\pi-4\le\econn\le3\pi$, so that
$\econn$ is a positive absolute constant.
The $S$-arc of \cref{lem:connector} is one such connector, and the whole of its $O(1)$ cost is of this
nature.

Cutting a competitor at two parameters where $\sweepdens$ has opposite signs and $h$, $h'$ are small, and
capping the two ends onto $\partial\disk$, produces a member of $\conncl$; carried out with the tolerances
matched to the estimates of \cref{sec:lower}, this yields
$\minen_{L}\ge L+c\,L^{4/9}+\econn-o(1)$ with an unspecified $c>0$ --- the interior connector contributes
\emph{additively}, by an amount independent of $L$.
We have not included the argument.
It is incomparable with \cref{thm:sharpmain}, trading the sharp constant for the additive term, and
$\econn=O(1)$ lies far below the error $O(L^{1/3})$ there; the natural home for it is a treatment that also
resolves $\ctwo$ and $c_{0}$, where the constant term becomes visible and $\econn$ is one of its three
ingredients.
Even the value of $\econn$ is then in question: the minimiser, if it exists, is a confined elastica with
multiplier $-1$, embedded because such an elastica is a graph over a line, but existence is not clear, since
minimising sequences may spiral just inside $\partial\disk$ at no cost, so that compactness has to be
established modulo arcs of $\partial\disk$.

\subsection*{Three dimensions}

The contrast with the Willmore problem in the ball \cite{MullerRoger2014} recorded in \cref{sec:intro} is
worth restating in the light of \cref{thm:sharpmain}.
There the second-order term is bounded, because nested spheres may be joined by catenoidal necks at no cost;
here the corresponding operation is a fold, and \cref{lem:fold} shows that a fold confined to a strip of width
$\Delta$ costs exactly $\dgconst^{2}/\Delta$.
The whole of $\cstar L^{4/9}$ is the price of that obstruction.

%% file: sec7-tool-disclosure.tex
\section{Tool and computational resource disclosure}
\label{sec:ai}

Following the recommendations of the Leiden Declaration on Artificial Intelligence and Mathematics
(2 June 2026), endorsed by the International Mathematical Union, I disclose the following.

\emph{Tools.}
The mathematical content of this paper was developed in collaboration with large language models
developed by Anthropic, principally Claude Opus~5 (\texttt{claude-opus-5}) and Claude Fable~5.1
(\texttt{claude-fable-5-1}), used between end of July and beginning of September~2026 via \textit{claude.ai}.
The models were run with three mininal author-supplied instruction files directing, respectively, the
line-by-line refereeing of proofs, the search for proof strategies, and the preparation of the
\LaTeX{} source.

Large language models developed by Alibaba Cloud, principally {\texttt{QWEN3.8-Max}} accessed via \textit{chat.qwen.ai/}, 
have been used to streamline some parts of the final version of the paper. 

\emph{Division of labour.}
My own contribution consists of the following: framing the problem; identifying the relevant
literature, and in particular the result from which the optimal constant is obtained; proposing the
existence of a fold on the outer envelope of the wire;  proposing to use the density argument proved in 
Lemma~\ref{lem:density}; and rewriting the statement of the main theorem in its present form.
All remaining mathematical content was generated by the models.
This includes, in particular, the definition of a fold and the proof that one may be selected on
the outer envelope (\cref{sec:lower-envelope}); the four energy budgets of \cref{prop:budgets} and
the observation that all four issue from the single closedness identity \cref{eq:closedness}; the
competitor of \cref{sec:upper} and the balance fixing $\delta\asymp L^{-4/9}$; the exact identity
of \cref{lem:identity}; and the write-up of the whole.
Several intermediate claims proposed during that process turned out to be false, and were
discovered to be so either by the models themselves or by me; the arguments presented here are
those that survived.
I have verified every statement and every proof in the present version.

\emph{Attribution.}
Automated tools are known to reproduce mathematical ideas without crediting their sources, and I
regard the resulting obligation to search actively for antecedents as falling on me.
I have therefore looked independently for prior occurrences of the arguments used here.

\emph{References.}
Every work cited in this paper has been located and read by me, and every attribution of a result
to a cited work has been checked against that work.

\emph{Responsibility.}
Responsibility for the correctness and adequacy of the arguments and results, and for the accuracy
and completeness of the citations to relevant prior work, rests exclusively with me.

%% file: appA-computations.tex
\section{Elementary computations}
\label{sec:appA}

We first record the approximation lemma behind \cref{rem:avoidorigin}: curves through the origin are
negligible for every infimum and every lower bound considered in this paper.

\begin{lemma}
  \label{lem:density}
  Let $\gamma\in\emb_{L}$ with $0\in\gamma([0,L])$ and let $\eta>0$.
  Then there exists $\gamma'\in\emb_{L}$ with $0\notin\gamma'([0,L])$ and
  $\energy(\gamma')\le\energy(\gamma)+\eta$.
\end{lemma}

\begin{proof}
  \emph{A graph window at the origin.}
  By injectivity there is exactly one $s_{0}$ with $\gamma(s_{0})=0$; rotate coordinates so that
  $\tau(s_{0})=e_{1}$.
  Write $\gamma=(x,y)$.
  Since $\gamma\in W^{2,2}\subset C^{1}$ and $x'(s_{0})=1$, there is $a>0$ with $x'\ge\tfrac12$ on
  $J_{0}:=[s_{0}-a,s_{0}+a]$, so $x|_{J_{0}}$ is a bi-Lipschitz $C^{1}$ change of variable and
  $\gamma(J_{0})$ is the graph of $f:=y\circ (x|_{J_{0}})^{-1}\in W^{2,2}$, with $f(0)=f'(0)=0$ and
  $\abs{f'}\le2$.
  The compact set $\gamma([0,L]\setminus \mathring J_{0})$ does not contain $0$, so it has positive distance
  from $0$; hence there is a closed rectangle $Q\subset\disk$ centred at $0$, with sides parallel to the axes
  and horizontal extent $[-b,b]$, such that
  \[
    \gamma([0,L])\cap Q=\bigl\{(t,f(t)):\abs t\le b\bigr\} .
  \]

  \emph{The perturbation.}
  Fix $\phi,\chi\in C^{\infty}_{c}(-b,b)$ with $\phi(0)=1$, $\chi(0)=0$, $\chi\not\equiv0$, and for small
  $\varepsilon,\mu$ let $\tilde\gamma_{\varepsilon,\mu}\in W^{2,2}$ be the closed curve obtained from $\gamma$
  by replacing the graph of $f$ over $[-b,b]$ with the graph of $f_{\varepsilon,\mu}:=f+\varepsilon\phi+\mu\chi$
  and reparametrising by arclength.
  For $\abs\varepsilon+\abs\mu$ small the new graph stays in $Q$, so $\tilde\gamma_{\varepsilon,\mu}$ is
  confined, closed, of class $W^{2,2}$ (the perturbation vanishes to all orders at $\pm b$), and embedded: it
  is injective on the graph, unchanged outside $Q$, and inside $Q$ only the graph is present.
  Its trace meets $0$ only if $f_{\varepsilon,\mu}(0)=\varepsilon$ vanishes; fix $\varepsilon\ne0$.

  \emph{Prescribing the length.}
  With $\varepsilon$ fixed, the length
  \[
    g(\mu):=\ell\bigl(\tilde\gamma_{\varepsilon,\mu}\bigr)
    =\bigl(L-\int_{-b}^{b}\sqrt{1+(f')^{2}}\,dt\bigr)
    +\int_{-b}^{b}\sqrt{1+(f'+\varepsilon\phi'+\mu\chi')^{2}}\,dt
  \]
  is of class $C^{2}$ in $\mu$ with
  $g''(\mu)=\int_{-b}^{b}\chi'^{2}\bigl(1+(f'+\varepsilon\phi'+\mu\chi')^{2}\bigr)^{-3/2}dt\ge c_{0}>0$
  uniformly for $\abs\varepsilon+\abs\mu\le1$, because $\abs{f'}\le2$.
  Hence $g(\mu)\ge g(0)+g'(0)\mu+\tfrac{c_{0}}2\mu^{2}$, and choosing the sign of $\mu$ equal to that of
  $g'(0)$ and $\abs\mu:=\bigl(2\abs{g(0)-L}/c_{0}\bigr)^{1/2}$ gives $g(\mu)\ge L$; since
  $\abs{g(0)-L}\le C\abs\varepsilon$, also $\mu=\mu(\varepsilon)\to0$ as $\varepsilon\to0$.

  \emph{Rescaling.}
  Set $\lambda:=L/g(\mu(\varepsilon))\in(0,1]$ and let $\gamma'$ be the arclength parametrisation of
  $\lambda\,\tilde\gamma_{\varepsilon,\mu(\varepsilon)}$.
  Since the length depends continuously on $(\varepsilon,\mu)$ and equals $L$ at $(0,0)$, while
  $\mu(\varepsilon)\to0$, one has $g(\mu(\varepsilon))\to L$ and hence $\lambda\to1$ as $\varepsilon\to0$.
  Then $\gamma'\in\emb_{L}$: its length is $L$, it lies in $\lambda\overline\disk\subset\overline\disk$, it is
  embedded, and it is positively oriented for small $\varepsilon$, the orientation being locally constant
  along a continuous family of embeddings.
  It avoids the origin, since the homothety fixes $0$.
  Finally, $\energy(\lambda\tilde\gamma)=\lambda^{-1}\energy(\tilde\gamma)$ and
  $\energy(\tilde\gamma_{\varepsilon,\mu})\to\energy(\gamma)$ as $(\varepsilon,\mu)\to0$, because the two
  curves differ only over $[-b,b]$, where the energy is
  $\int_{-b}^{b}f_{\varepsilon,\mu}''^{2}\,\bigl(1+(f_{\varepsilon,\mu}')^{2}\bigr)^{-5/2}dt$ and
  $f_{\varepsilon,\mu}\to f$ in $W^{2,2}$.
  Choosing $\varepsilon$ small enough that $\energy(\gamma')\le\energy(\gamma)+\eta$ completes the proof.
\end{proof}

We next record the expansion used in \cref{sec:upper}.

\begin{proof}[Proof of \cref{lem:graphexcess}]
  For the polar graph $\theta\mapsto(1-u(\theta))e^{i\theta}$ put
  \[
    a:=(1-u)^{2}+(u')^{2},\qquad b:=(u')^{2}+(1-u)u'',\qquad n:=a+b .
  \]
  Then $ds=a^{1/2}d\theta$ and the curvature of a polar graph $r=r(\theta)$, namely
  $\kappa=(r^{2}+2(r')^{2}-rr'')(r^{2}+(r')^{2})^{-3/2}$, becomes $\kappa=n\,a^{-3/2}$.
  Hence
  \[
    \kappa^{2}\,ds-ds=\frac{n^{2}-a^{3}}{a^{5/2}}\,d\theta
    =\frac{a^{2}(1-a)+2ab+b^{2}}{a^{5/2}}\,d\theta .
  \]
  Since $1-a=2u-u^{2}-(u')^{2}$ and $\norm{u}_{\infty}+\norm{u'}_{\infty}\le\tfrac18$, we have
  $\tfrac12\le a\le 2$, so $a^{-5/2}=1+O\bigl(\abs u+(u')^{2}\bigr)$ and the numerator equals
  \[
    2u+2u''+(u'')^{2}+O\bigl(u^{2}+(u')^{2}+\abs{u\,u''}+\abs{u''}^{3}\bigr),
  \]
  where we used $b=u''+O\bigl((u')^{2}+\abs{u\,u''}\bigr)$ and
  $b^{2}=(u'')^{2}+O\bigl(u^{2}+(u')^{2}+\abs{u''}^{3}\bigr)$, the term $-2u(u'')^{2}$ arising in $b^{2}$ being
  admissible by Young's inequality, $\abs u(u'')^{2}\le\tfrac13\abs{u}^{3}+\tfrac23\abs{u''}^{3}\le\tfrac1{24}u^{2}+\tfrac23\abs{u''}^{3}$, using $\abs u\le\tfrac18$.
  Multiplying the two expansions and integrating over $I$ gives the claim: the products of the error terms with
  $2u+2u''+(u'')^{2}$ are of the same admissible form, again by Young, through
  $\abs{u''}(u')^{2}\le\tfrac13\abs{u''}^{3}+\tfrac1{12}(u')^{2}$ and the inequality just displayed.
  This is why $\abs{u''}^{3}$ must appear in the error list, even though it is of higher order than the other
  three terms: it is what absorbs the cross terms.
\end{proof}

%% file: refs.bib
@article{BellettiniMugnai2004,
  author  = {Bellettini, Giovanni and Mugnai, Luca},
  title   = {Characterization and representation of the lower semicontinuous envelope of the elastica functional},
  journal = {Annales de l'Institut Henri Poincar\'e C, Analyse non lin\'eaire},
  volume  = {21},
  number  = {6},
  pages   = {839--880},
  year    = {2004}
}

@article{bellettini2007varifolds,
  title={A varifolds representation of the relaxed elastica functional},
  author={Bellettini, Giovanni and Mugnai, Luca},
  journal={Journal of Convex Analysis},
  volume={14},
  number={3},
  pages={543},
  year={2007},
  publisher={HELDERMANN VERLAG LANGER GRABEN 17, 32657 LEMGO, GERMANY}
}

@article{BellettiniDalMasoPaolini1993,
  author  = {Bellettini, Giovanni and Dal Maso, Gianni and Paolini, Maurizio},
  title   = {Semicontinuity and relaxation properties of a curvature depending functional in {2D}},
  journal = {Annali della Scuola Normale Superiore di Pisa, Classe di Scienze},
  volume  = {20},
  number  = {2},
  pages   = {247--297},
  year    = {1993}
}

@article{delladio1997special,
  title={Special generalized Gauss graphs and their application to minimization of functionals involving curvatures},
  author={Delladio, Silvano},
  journal={Journal fur die Reine und Angewandte Mathematik},
  pages={17--44},
  year={1997}
}

@article{DondlMugnaiRoger2011,
  author  = {Dondl, Patrick W. and Mugnai, Luca and R\"oger, Matthias},
  title   = {Confined elastic curves},
  journal = {SIAM Journal on Applied Mathematics},
  volume  = {71},
  number  = {6},
  pages   = {2205--2226},
  year    = {2011}
}

@article{Wojtowytsch2018,
    title = {Confined elasticae and the buckling of cylindrical shells},
    author = {Wojtowytsch, Stephan},
    pages = {555--587},
    volume = {14},
    number = {4},
    journal = {Advances in Calculus of Variations},
    year = {2021},
}

@article{DayrensMasnouNovaga2015,
  title={Existence, regularity and structure of confined elasticae},
  author={Dayrens, Francois and Masnou, Simon and Novaga, Matteo},
  journal={ESAIM: Control, Optimisation and Calculus of Variations},
  volume={24},
  number={1},
  pages={25--43},
  year={2018}
}

@article{MuellerRupp2021,
  title={A Li-Yau inequality for the 1-dimensional Willmore energy},
  author={M\"uller, Marius and Rupp, Fabian},
  journal={Advances in Calculus of Variations},
  volume={16},
  number={2},
  pages={337--362},
  year={2023},
  publisher={De Gruyter}
}

@article{MullerRoger2014,
  author  = {M\"uller, Stefan and R\"oger, Matthias},
  title   = {Confined structures of least bending energy},
  journal = {Journal of Differential Geometry},
  volume  = {97},
  number  = {1},
  pages   = {109--139},
  year    = {2014},
}

@article{Boue2006,
  author  = {Bou\'e, Laurent and Adda-Bedia, Mokhtar and Boudaoud, Arezki and Cassani, Davide and Couder, Yves and Eddi, Antonin and
             Trejo, Miguel},
  title   = {Spiral patterns in the packing of flexible structures},
  journal = {Physical Review Letters},
  volume  = {97},
  number  = {16},
  pages   = {166104},
  year    = {2006}
}

@article{Alben2021,
  title={Packing of elastic rings with friction},
  author={Alben, Silas},
  journal={Proceedings: Mathematical, Physical and Engineering Sciences},
  volume={478},
  number={2258},
  pages={1--20},
  year={2022}
}

@article{Deboeuf2024,
  author  = {Deboeuf, St\'ephanie and Proti\`ere, Suzie and Katzav, Eytan},
  title   = {Yin-yang spiraling transition of a confined buckled elastic sheet},
  journal = {Physical Review Research},
  volume  = {6},
  number  = {1},
  pages   = {013100},
  year    = {2024}
}

@article{Abramian2026,
  author  = {Abramian, Ana\"is and De Brion, Matthieu and Katzav, Eytan and Deboeuf, St\'ephanie},
  title   = {A force master curve for the wrinkle-to-spiral transition of confined sheets},
  journal = {Preprint hal-05663157},
  year    = {2026}
}

@article{Wang2026,
  author  = {Wang, Jiayu and Lachat, Sara and Abramian, Ana\"is and Deboeuf, St\'ephanie and Neukirch,
             S\'ebastien},
  title   = {Universal response curves for the confined elastica},
  journal={Journal of the Mechanics and Physics of Solids},
  year    = {2026},
  pages={106778},
  publisher={Elsevier}
}

@article{DeckelnickGrunau2007,
  author  = {Deckelnick, Klaus and Grunau, Hans-Christoph},
  title   = {Boundary value problems for the one-dimensional Willmore equation},
  journal = {Calculus of Variations and Partial Differential Equations},
  volume  = {30},
  number  = {3},
  pages   = {293--314},
  year    = {2007},
}

@article{MiuraYoshizawa2024,
  author  = {Miura, Tatsuya and Yoshizawa, Kensuke},
  title   = {Complete classification of planar $p$-elasticae},
  journal = {Annali di Matematica Pura ed Applicata},
  year    = {2024},
  note    = {arXiv:2203.08535},
}

@book{Newman1951,
  author    = {Newman, M. H. A.},
  title     = {Elements of the Topology of Plane Sets of Points},
  edition   = {Second},
  publisher = {Cambridge University Press},
  address   = {Cambridge},
  year      = {1951},
}

@article{Gale1979,
  author  = {Gale, David},
  title   = {The game of {H}ex and the {B}rouwer fixed-point theorem},
  journal = {Amer. Math. Monthly},
  volume  = {86},
  year    = {1979},
  number  = {10},
  pages   = {818--827},
}

@book{Hatcher2002,
  author    = {Hatcher, Allen},
  title     = {Algebraic Topology},
  publisher = {Cambridge University Press},
  address   = {Cambridge},
  year      = {2002},
}

@article{BourgainBrezisMironescu2000,
  author  = {Bourgain, Jean and Brezis, Ha{\"\i}m and Mironescu, Petru},
  title   = {Lifting in {S}obolev spaces},
  journal = {J. Anal. Math.},
  volume  = {80},
  year    = {2000},
  pages   = {37--86},
}

@article{Chakerian1962,
  author  = {Chakerian, Gulbank D.},
  title   = {An inequality for closed space curves},
  journal = {Pacific Journal of Mathematics},
  volume  = {12},
  number  = {1},
  pages   = {53--57},
  year    = {1962},
}

@incollection{Sullivan2008,
  author    = {Sullivan, John M.},
  title     = {Curves of finite total curvature},
  booktitle = {Discrete Differential Geometry},
  editor    = {Bobenko, Alexander I. and Schr\"oder, Peter and Sullivan, John M. and Ziegler, G\"unter M.},
  series    = {Oberwolfach Seminars},
  volume    = {38},
  pages     = {137--161},
  publisher = {Birkh\"auser},
  address   = {Basel},
  year      = {2008},
  note      = {arXiv:math/0606007},
}

@article{Fary1950,
  author  = {F\'ary, Istv\'an},
  title   = {Sur certaines in\'egalit\'es g\'eom\'etriques},
  journal = {Acta Universitatis Szegediensis. Acta Scientiarum Mathematicarum},
  volume  = {12},
  pages   = {117--124},
  year    = {1950},
}

@article{Miura2021,
  author  = {Miura, Tatsuya},
  title   = {Polar tangential angles and free elasticae},
  journal = {Mathematics in Engineering},
  volume  = {3},
  number  = {4},
  pages   = {Paper No. 034, 12 pp.},
  year    = {2021},
}

@article{CaffarelliFriedman1979,
  author  = {Caffarelli, Luis A. and Friedman, Avner},
  title   = {The obstacle problem for the biharmonic operator},
  journal = {Annali della Scuola Normale Superiore di Pisa, Classe di Scienze (4)},
  volume  = {6},
  number  = {1},
  pages   = {151--184},
  year    = {1979},
}

@article{DallAcquaDeckelnick2018,
  author  = {Dall'Acqua, Anna and Deckelnick, Klaus},
  title   = {An obstacle problem for elastic graphs},
  journal = {SIAM Journal on Mathematical Analysis},
  volume  = {50},
  number  = {1},
  pages   = {119--137},
  year    = {2018},
}

@article{Mueller2019,
  author  = {M\"uller, Marius},
  title   = {An obstacle problem for elastic curves: existence results},
  journal = {Interfaces and Free Boundaries},
  volume  = {21},
  number  = {1},
  pages   = {87--129},
  year    = {2019},
}

@article{Yoshizawa2021,
  author  = {Yoshizawa, Kensuke},
  title   = {A remark on elastic graphs with the symmetric cone obstacle},
  journal = {SIAM Journal on Mathematical Analysis},
  volume  = {53},
  number  = {2},
  pages   = {1857--1885},
  year    = {2021},
}

@article{DallAcquaMuellerOkabeYoshizawa2024,
  author  = {Dall'Acqua, Anna and M\"uller, Marius and Okabe, Shinya and Yoshizawa, Kensuke},
  title   = {An obstacle problem for the $p$-elastic energy},
  journal = {Calculus of Variations and Partial Differential Equations},
  volume  = {63},
  number  = {6},
  pages   = {Paper No. 145, 43 pp.},
  year    = {2024},
}

@article{Miura2016,
  author  = {Miura, Tatsuya},
  title   = {Singular perturbation by bending for an adhesive obstacle problem},
  journal = {Calculus of Variations and Partial Differential Equations},
  volume  = {55},
  number  = {1},
  pages   = {Art. 19, 24 pp.},
  year    = {2016},
}

@article{Miura2023,
  author  = {Miura, Tatsuya},
  title   = {{L}i--{Y}au type inequality for curves in any codimension},
  journal = {Calculus of Variations and Partial Differential Equations},
  volume  = {62},
  number  = {8},
  pages   = {Paper No. 216, 28 pp.},
  year    = {2023},
}

@article{LangerSinger1984,
  author  = {Langer, Joel and Singer, David A.},
  title   = {The total squared curvature of closed curves},
  journal = {Journal of Differential Geometry},
  volume  = {20},
  number  = {1},
  pages   = {1--22},
  year    = {1984},
}

@article{AvvakumovKarpenkovSossinsky2013,
  author  = {Avvakumov, Sergey and Karpenkov, Oleg and Sossinsky, Alexei},
  title   = {{E}uler elasticae in the plane and the {W}hitney--{G}raustein theorem},
  journal = {Russian Journal of Mathematical Physics},
  volume  = {20},
  number  = {3},
  pages   = {257--267},
  year    = {2013},
}

@article{MantegazzaPludaPozzetta2021,
  author  = {Mantegazza, Carlo and Pluda, Alessandra and Pozzetta, Marco},
  title   = {A survey of the elastic flow of curves and networks},
  journal = {Milan Journal of Mathematics},
  volume  = {89},
  number  = {1},
  pages   = {59--121},
  year    = {2021},
}

@article{BartelsWeyer2022,
  author  = {Bartels, S\"oren and Weyer, Pascal},
  title   = {Computing confined elasticae},
  journal = {Advances in Continuous and Discrete Models},
  volume  = {2022},
  pages   = {Paper No. 58},
  year    = {2022},
}

@article{GerlachReiterVonDerMosel2017,
  author  = {Gerlach, Henryk and Reiter, Philipp and von der Mosel, Heiko},
  title   = {The elastic trefoil is the doubly covered circle},
  journal = {Archive for Rational Mechanics and Analysis},
  volume  = {225},
  number  = {1},
  pages   = {89--139},
  year    = {2017},
}

@article{GerlachVonDerMosel2011,
  author  = {Gerlach, Henryk and von der Mosel, Heiko},
  title   = {What are the longest ropes on the unit sphere?},
  journal = {Archive for Rational Mechanics and Analysis},
  volume  = {201},
  number  = {1},
  pages   = {303--342},
  year    = {2011},
}

@article{BoueKatzav2007,
  author  = {Bou\'e, Laurent and Katzav, Eytan},
  title   = {Folding of flexible rods confined in {2D} space},
  journal = {Europhysics Letters (EPL)},
  volume  = {80},
  number  = {5},
  pages   = {54002},
  year    = {2007},
}

@article{StoopWittelHerrmann2008,
  author  = {Stoop, Norbert and Wittel, Falk K. and Herrmann, Hans J.},
  title   = {Morphological phases of crumpled wire},
  journal = {Physical Review Letters},
  volume  = {101},
  number  = {9},
  pages   = {094101},
  year    = {2008},
}

@article{DonatoGomesDeSouza2003,
  author  = {Donato, Cassia C. and Gomes, Marcelo A. F. and de Souza, Rubens E.},
  title   = {Scaling properties in the packing of crumpled wires},
  journal = {Physical Review E},
  volume  = {67},
  number  = {2},
  pages   = {026110},
  year    = {2003},
}

@article{VetterWittelStoopHerrmann2013,
  author  = {Vetter, Roman and Wittel, Falk K. and Stoop, Norbert and Herrmann, Hans J.},
  title   = {Finite element simulation of dense wire packings},
  journal = {European Journal of Mechanics A/Solids},
  volume  = {37},
  pages   = {160--171},
  year    = {2013},
}

@article{Odijk2006,
  author  = {Odijk, Theo},
  title   = {{DNA} confined in nanochannels: hairpin tightening by entropic depletion},
  journal = {The Journal of Chemical Physics},
  volume  = {125},
  number  = {20},
  pages   = {204904},
  year    = {2006},
}

@article{Chakerian1964,
  author  = {Chakerian, Gulbank D.},
  title   = {On some geometric inequalities},
  journal = {Proceedings of the American Mathematical Society},
  volume  = {15},
  number  = {6},
  pages   = {886--888},
  year    = {1964},
}
